\documentclass{amsart}
\usepackage{amssymb,amsmath}
\usepackage{amsfonts}
\usepackage{color,cite}
\usepackage[normalem]{ulem}

\newtheorem{theorem}{Theorem}
\theoremstyle{plain}

\newtheorem{definition}{Definition}

\newtheorem{lemma}{Lemma}

\newtheorem{proposition}{Proposition}
\newtheorem{remark}{Remark}

\numberwithin{equation}{section}

\input{tcilatex}

\begin{document}
\title{The Brezis-Nirenberg problem for fractional elliptic operators}
\author{Ko-Shin Chen}
\address{Department of Mathematics\\
University of Connecticut\\
Storrs, CT 06269 USA}
\email{ko-shin.chen@uconn.edu}
\urladdr{}
\thanks{}
\author{Marcos Montenegro}
\address{Departamento de Matem\'{a}tica \\
Universidade Federal de Minas Gerais\\
Belo Horizonte, 30123-970 Brazil}
\email{montene@mat.ufmg.br}
\author{Xiaodong Yan}
\address{Department of Mathematics\\
University of Connecticut\\
Storrs, CT 06269 USA}
\email{xiaodong.yan@uconn.edu}
\thanks{}
\date{\today }

\begin{abstract}
Let $\mathcal{L} = \mathrm{div}(A(x) \nabla )$ be a uniformly elliptic
operator in divergence form in a bounded open subset $\Omega$ of $\mathbb{R}%
^n$. We study the effect of the operator $\mathcal{L}$ on the existence and
nonexistence of positive solutions of the nonlocal Brezis-Nirenberg problem

\begin{equation*}
\left\{
\begin{array}{rcll}
\displaystyle(-\mathcal{L})^{s}u & = & u^{\frac{n+2s}{n-2s}}+\lambda u & %
\mbox{in}\ \Omega , \\
u & = & 0 & \mbox{on}\ \mathbb{R}^{n}\setminus \Omega \\
&  &  &
\end{array}%
\right.
\end{equation*}%
where $(-\mathcal{L})^{s}$ denotes the fractional power of $-\mathcal{L}$
with zero Dirichlet boundary values on $\partial \Omega $, $0<s<1$, $n>2s$
and $\lambda $ is a real parameter. By assuming $A(x) \geq A(x_{0})$ for all
$x \in \overline{\Omega }$ and $A(x) \leq A(x_{0})+|x-x_{0}|^{\sigma }I_{n}$
near some point $x_{0}\in \overline{\Omega }$, we prove existence theorems
for any $\lambda \in (0,\lambda_{1,s}(-\mathcal{L}))$, where $\lambda_{1,s}(-%
\mathcal{L})$ denotes the first Dirichlet eigenvalue of $(-\mathcal{L})^{s}$%
. Our existence result holds true for $\sigma >2s$ and $n\geq 4s$ in the
interior case ($x_{0}\in \Omega $) and for $\sigma >\frac{2s(n-2s)}{n-4s}$
and $n > 4s$ in the boundary case ($x_{0}\in \partial \Omega$). Nonexistence
for star-shaped domains is obtained for any $\lambda \leq 0$.
\end{abstract}

\maketitle

\section{ \ Introduction and statements}

A lot of attention has been paid to a number of counterparts of the
Brezis-Nirenberg problem since the pioneer paper \cite{BN} which consists in
determining all values of $\lambda $ for which the problem

\begin{equation}  \label{BN}
\left\{
\begin{array}{rcll}
- \Delta u & = & u^{\frac{n+2}{n-2}} + \lambda u & \mathrm{in} \ \ \Omega,
\\
u & = & 0 & \mathrm{on} \ \ \partial\Omega%
\end{array}
\right.
\end{equation}
admits a positive solution, where $\Omega$ is a bounded open subset of $%
\mathbb{R}^n$, $n \geq 3$ and $\lambda$ is a real parameter.

According to \cite{BN}, the problem (\ref{BN}) admits a positive solution
for any $\lambda \in (0,\lambda _{1}(-\Delta ))$ provided that $n\geq 4$,
where $\lambda _{1}(-\Delta )$ denotes the first Dirichlet eigenvalue of $%
-\Delta $ on $\Omega $. Moreover, the problem has no such solution for $%
n\geq 3$ if either $\lambda \geq \lambda _{1}(-\Delta )$ or $\lambda \leq 0$
and $\Omega $ is a star-shaped $C^{1}$ domain. When $n=3$ and $\Omega $ is a
ball, a positive solution of (\ref{BN}) exists if, and only if, $\lambda \in
(\frac{1}{4}\lambda _{1}(-\Delta ),\lambda _{1}(-\Delta ))$.

The Brezis-Nirenberg problem for uniformly elliptic operators in divergence
form has been studied in the works \cite{deMoMo,HE,HM}. Precisely, consider
the problem

\begin{equation}  \label{BN-g}
\left\{
\begin{array}{rcll}
- \mathcal{L} u & = & u^{\frac{n+2}{n-2}} + \lambda u & \mathrm{in} \ \
\Omega, \\
u & = & 0 & \mathrm{on} \ \ \partial\Omega%
\end{array}
\right.
\end{equation}
where $\mathcal{L} = \mathrm{div}(A(x) \nabla )$. Assume that $A(x) =
(a_{ij}(x))$ is a positive definite symmetric matrix for each $x \in
\overline{\Omega}$ with continuous entries on $\overline{\Omega}$, so that $%
\mathcal{L}$ is a selfadjoint uniformly elliptic operator. Assume also there
exist a point $x_0 \in \overline{\Omega}$ and a constant $C_0 > 0$ such that

\begin{equation}
A(x) \geq A(x_0) \ \ \mathrm{for\ every} \ x \in \overline{\Omega}
\label{H2}
\end{equation}
and

\begin{equation}  \label{H1}
A(x) \leq A(x_0) + C_0 |x - x_0|^\sigma I_n \ \ \mathrm{locally\ around} \
x_0
\end{equation}
both in the sense of bilinear forms, where $I_n$ denotes the $n \times n$
identity matrix. In \cite{HE}, Egnell focused on the interior case ($x_0 \in
\Omega$) and proved that problem (\ref{BN-g}) admits a positive solution for
any $\lambda \in (0, \lambda_1(-\mathcal{L}))$ provided that $n \geq 4$ and $%
\sigma > 2$, where $\lambda_1(-\mathcal{L})$ denotes the first Dirichlet
eigenvalue of $-\mathcal{L}$ on $\Omega$. The boundary case ($x_0 \in
\partial \Omega$) has recently been treated in \cite{HM}, which proves the
existence of a positive solution for any $\lambda \in (0, \lambda_1(-%
\mathcal{L}))$ provided that $n > 4$, $\sigma > \frac{2n-4}{n-4}$ and the
boundary of $\Omega$ is $\alpha$-singular at $x_0$, with $\alpha \in [1,
\sigma \frac{n - 4}{2n - 4})$, in the following sense:

\begin{definition}
\label{alphasingular} \label{boundary regularity} The boundary of an open
subset $\Omega$ of $\mathbb{R}^n$ is said to be $\alpha$-singular at the
point $x_0$, with $\alpha \geq 1$, if there exist a constant $\delta > 0$
and a sequence $(x_j) \subset \Omega$ such that $x_j \to x_0$ as $j
\rightarrow + \infty$ and $B(x_j, \delta |x_j - x_0|^\alpha) \subseteq \Omega
$.
\end{definition}

The nonexistence of positive solutions of (\ref{BN-g}) for $\lambda \leq 0$
on star-shaped domains has been proved in \cite{HE} (see also \cite{deMoMo})
by assuming $a_{ij}\in C^{1}(\overline{\Omega }\setminus \{x_{0}\})$ such
that $a_{ij}^{\prime }(x):=\nabla a_{ij}(x)\cdot (x-x_{0})$ extends
continuously to $x_{0}$ and $A^{\prime }(x)=(a_{ij}^{\prime }(x))$ is
positive semi-definite for every $x\in \Omega $. Nonexistence of positive
solution in the case $\lambda \geq \lambda _{1}(-\mathcal{L})$ follows from
a standard argument.

When $0<s<1$, Caffarelli and Silvestre \cite{CaSi} introduced the
characterization of the fractional power of the Laplace operator $(-\Delta)^s
$ in terms of a Dirichlet-to-Neumann map associated to a suitable extension
problem. Since then, a great deal of attention has been dedicated in the
last years to nonlinear nonlocal problems involving this operator. See for
example \cite{BCPS, BrCPS, CT, CDDS, choi, CK, T, yang}, among others. Two
of them (see \cite{T} for $s = \frac12$ and \cite{BCPS} for other values of $%
s \in (0,1)$) consider the following counterpart of the Brezis-Nirenberg
problem

\begin{equation}  \label{BN-frac}
\left\{
\begin{array}{rcll}
(- \Delta)^s u & = & u^{\frac{n+2s}{n-2s}} + \lambda u & \mathrm{in} \ \
\Omega, \\
u & = & 0 & \mathrm{in} \ \ \mathbb{R}^n \setminus \Omega%
\end{array}
\right.
\end{equation}
In particular, it has been proved that problem (\ref{BN-frac}) admits a
positive viscosity solution for any $\lambda \in (0, \lambda_{1,s}\left(-\Delta
\right))$ provided that $n \geq 4s$, where $\lambda_{1,s}\left(-\Delta\right)$
denotes the first Dirichlet eigenvalue of $(-\Delta)^s$ on $\Omega$.
Moreover, there exists no such solution in $C^1(\overline{\Omega})$ for $n >
2s$ if either $\lambda \leq 0$ and $\Omega$ is a star-shaped $C^1$ domain or
$\lambda \geq \lambda_{1,s}\left(-\Delta \right)$.

This work dedicates special attention to the effect of the elliptic operator
$\mathcal{L}$ (or of the matrix $A(x)$) on the existence and nonexistence of
positive viscosity solutions of the following Brezis-Nirenberg problem
involving the fractional power of $-\mathcal{L}$,

\begin{equation}  \label{fractionalbn}
\left\{
\begin{array}{rcll}
(- \mathcal{L})^s u & = & u^{\frac{n+2s}{n-2s}} + \lambda u & \mathrm{in} \
\ \Omega, \\
u & = & 0 & \mathrm{in} \ \ \mathbb{R}^n \setminus \Omega%
\end{array}
\right.
\end{equation}

Denote by $\lambda_{1,s}\left(-\mathcal{L} \right)$ the first Dirichlet
eigenvalue of $(-\mathcal{L})^s$ on $\Omega$.

The main existence theorems are:

{ }

\begin{theorem}
\label{T.1} \textrm{(Interior\ case)} Let $0<s<1$ and $\Omega \subset
\mathbb{R}^{n}$ be a bounded open set. Assume entries of the matrix $A$ are
continuous in $\overline{\Omega }$ and $A$ satisfies (\ref{H2}), (\ref{H1})
for some $x_{0}\in \Omega $. Then (\ref{fractionalbn}) admits at least one
positive weak solution for any $\lambda \in \left( 0,\lambda _{1,s}\left( -%
\mathcal{L}\right) \right) $ provided $n\geq 4s$ and $\sigma >2s$. If $%
\partial \Omega $ is of $C^{1,1}$ class and each entry of $A(x)$ belongs to $%
C^{1}(\overline{\Omega })$, then our weak solution $u$ belongs to $%
C^{0,\alpha }(\overline{\Omega })$ if $0<s<1/2$ and to $C^{1,\alpha }(%
\overline{\Omega })$ if $1/2\leq s<1$ for any $0<\alpha <1$.
\end{theorem}

{ }

\begin{theorem}
\label{T.2} \textrm{(Boundary case)} Let $0<s<1$ and $\Omega \subset \mathbb{%
R}^{n}$ be a bounded open set. Assume entries of the matrix $A$ are
continuous in $\overline{\Omega }$ and $A$ satisfies (\ref{H2}), (\ref{H1})
for some $x_{0}$ on $\partial \Omega $. Suppose $\partial \Omega $ is $%
\alpha $-singular at $x_{0}$. Then (\ref{fractionalbn}) admits at least one
positive weak solution for any $\lambda \in \left( 0,\lambda _{1,s}\left( -%
\mathcal{L}\right) \right) $ if $n>4s$, $\sigma >\frac{2s(n-2s)}{n-4s}$ and $%
1\leq \alpha <\frac{\sigma (n-4s)}{2s(n-2s)}$. If $\partial \Omega $ is of $%
C^{1,1}$ class and each entry of $A(x)$ belongs to $C^{1}(\overline{\Omega })
$, then our weak solution $u$ belongs to $C^{0,\alpha }(\overline{\Omega })$
if $0<s<1/2$ and to $C^{1,\alpha }(\overline{\Omega })$ if $1/2\leq s<1$ for
any $0<\alpha <1$.
\end{theorem}

The nonexistence theorem states that

{ }

\begin{theorem}
\label{T.3} Let $0<s<1$, $n>2s$, and $\Omega \subset \mathbb{R}^{n}$ be a
bounded open set which is star-shaped with respect to some point $x_{0}\in
\overline{\Omega }$ and its boundary is of $C^{1}$ class. Assume matrix $A$
satisfies \eqref{H2}-\eqref{H1} for some $x_{0}$ in $\overline{\Omega }$,
moreover, assume $a_{ij}\in C^{1}(\overline{\Omega }\setminus \{x_{0}\})$
and $a_{ij}^{\prime }(x):=\nabla a_{ij}(x)\cdot (x-x_{0})$ extends
continuously to $x_{0}$ and $A^{\prime }(x)=(a_{ij}^{\prime }(x))$ is
positive semi-definite for every $x\in \Omega $. Then (\ref{fractionalbn})
admits no positive solution in $C^{1}(\overline{\Omega })$ for any $\lambda
\leq 0$. Furthermore, if $\partial \Omega $ is of $C^{1,1}$ class and each
entry of $A(x)$ belongs to $C^{1}(\overline{\Omega })$, then (\ref%
{fractionalbn}) admits no positive weak solution for any $\lambda \leq 0$
provided that $s\geq 1/2$.
\end{theorem}

Theorems \ref{T.1} and \ref{T.3} extend the existence and nonexistence
results of \cite{BCPS} and \cite{T}, since the constant matrix $A(x) = I_n$
clearly fulfills our assumptions. All results in \cite{HE,HM} for $s = 1$
are also extended fully to any $0 < s < 1$. One prototype example of
operator $\mathcal{L}$ is $A(x) = A_0 + |x - x_0|^\sigma I_n$.

A natural approach for solving (\ref{fractionalbn}) consists in searching
for minimizers of the functional

\begin{equation*}
u\mapsto \int_{\mathbb{R}^{n}}|(-\mathcal{L})^{s/2}u|^{2}-\lambda
\int_{\Omega }u^{2}dx
\end{equation*}%
subject to
\begin{equation*}
\int_{\Omega }|u|^{\frac{2n}{n-2s}}dx=1\,.
\end{equation*}

However, fractional integrals of this type are generally difficult to be
handled directly. On the other hand, fractional powers of elliptic operators
in divergence form were recently described in \cite{ST} as
Dirichlet-to-Neumann maps for an extension problem in the spirit of the
extension problem for the fractional Laplace operator on $\mathbb{R}^{n}$ of
\cite{CaSi}. In section 2, we take some advantage of this description and
provide an equivalent variational formulation which will be used in our
proof of existence. We also present an existence tool and a regularity
result of weak solutions.

The existence of minimizers for the new constrained functional often relies
on the construction and estimates of suitable bubbles involving extremal
functions of Sobolev type inequalities. Although, this is a well known
strategy, new and important difficulties arise in present context. Indeed,
the most delicate part in the proof of Theorem \ref{T.1} (the interior case)
is caused by the term $\left\vert x-x_{0}\right\vert ^{\sigma }$ in the
inequality (\ref{H1}), which essentially involves estimates of the multiple
integral
\begin{equation*}
\int_{B_{R}(0)}\int_{0}^{\infty }|x|^{\sigma }y^{1-2s}|\nabla
_{x}w_{1}(x,y)|^{2}dydx
\end{equation*}%
on the whole ball of radius $R$ for $R>0$ large enough, where $w_{1}(x,y)$
is given by

\begin{equation*}
w_{1}(x,y):=c_{s} y^{2s}\int_{\mathbb{R}^{n}}\frac{u_{1}(\xi )}{(|x-\xi
|^{2}+y^{2})^{\frac{n+2s}{2}}}d\xi \,,
\end{equation*}
with $c_{s}$ being an appropriate normalization constant and

\begin{equation*}
u_{1}(\xi )=\frac{1}{(1+|\xi |^{2})^{\frac{n-2s}{2}}}\,.
\end{equation*}%
In Section 3, we estimate the integral mentioned above and, as a byproduct,
we prove Theorem \ref{T.1}. The bubbles used in the proof of Theorem 1 do
not work in the boundary case because we need to compare the least energy
level to the corresponding best trace constant in $\mathbb{R}_{+}^{n+1}$.
The idea for overcoming this difficulty is to consider suitable bubbles
concentrated in interior points converging fast to the boundary point $x_{0}$
in an appropriate way. The construction depends on the order $\alpha $ of
the singularity of the boundary at $x_{0}$. In Section 4 we introduce such
bubbles and derive the necessary estimates in the boundary case. Proof of
Theorem \ref{T.2} then follows. Finally, in Section 5, we establish a
Pohozaev identity for $C^{1}$ solutions of (\ref{fractionalbn}) and use it
to prove Theorem \ref{T.3}.

\section{The variational framework and main tools}

For the precise definition of the fractional power of the selfadjoint
elliptic operator $-\mathcal{L}$, we consider an orthonormal basis of $%
L^2(\Omega)$ consisting of eigenfunctions $\phi_k \in H^1_0(\Omega)$, $k =
1, 2, \cdots $, that correspond to eigenvalues $\lambda_1 < \lambda_2 \leq
\cdots$. The domain of $(-\mathcal{L})^s$, denoted here by $H^s$, $0 < s < 1$%
, is defined as the Hilbert space of functions $u = \sum_{k = 1}^\infty c_k
\phi_k \in L^2(\Omega)$ such that $\sum_{k = 1}^\infty \lambda_k^s c_k^2 <
\infty$ endowed with the inner product $\langle u, v \rangle_{H^s} :=
\sum_{k = 1}^\infty \lambda_k^s c_k d_k$, where $v = \sum_{k = 1}^\infty d_k
\phi_k$. For each $u = \sum_{k=1}^{\infty} c_{k} \phi_{k} \in H^s$, we
define $\left( -\mathcal{L}\right) ^{s}u=\sum_{k}c_{k}\lambda
_{k}^{s}\varphi_{k}$. With this definition, we have

\begin{equation*}
\langle u, v \rangle_{H^s} = \langle \left( -\mathcal{L}\right) ^{\frac{s}{2}%
} u, \left( -\mathcal{L}\right) ^{\frac{s}{2}} v \rangle_{L^{2}\left( \Omega
\right)}\, .
\end{equation*}

It is well known that (see details in \cite{CS, Li})
\begin{equation*}
H^s = \left\{
\begin{array}{lcll}
H^s(\Omega), &  & \mathrm{if} \ 0 < s < 1/2, &  \\
H_{00}^{1/2}(\Omega), &  & \mathrm{if} \ s = 1/2, &  \\
H_0^s(\Omega), &  & \mathrm{if} \ 1/2 < s < 1 &
\end{array}
\right.
\end{equation*}
The spaces $H^s(\Omega)$ and $H_0^s(\Omega)$, $s \neq 1/2$, are the
classical fractional Sobolev spaces given as the completion of $%
C_0^\infty(\Omega)$ under the norm
\begin{equation*}
\|u\|^2_{H^s(\Omega)} = \|u\|^2_{L^2(\Omega)} + [u]^2_{H^s(\Omega)}\, ,
\end{equation*}
where
\begin{equation*}
[u]^2_{H^s(\Omega)} = \int_\Omega \int_\Omega \frac{\left( u(x) - u(y)
\right)^2}{|x - y|^{n + 2s}} dx dy\, .
\end{equation*}
The space $H_{00}^{1/2}(\Omega)$ is the Lions-Magenes space which consists
of functions $u \in L^2(\Omega)$ such that $[u]^2_{H^{1/2}(\Omega)} < \infty$
and
\begin{equation*}
\int_{\Omega} \frac{u(x)^2}{\mathrm{dist}(x, \partial \Omega)} dx < \infty\,
.
\end{equation*}

The Hilbert space $H^s$ is compactly embedded in $L^q(\Omega)$ for any $1
\leq q < \frac{2n}{n-2s}$ and continuously in $L^{\frac{2n}{n-2s}}(\Omega)$
provided that $n > 2s$. So, a natural strategy to solve (\ref{fractionalbn})
in a weak sense consists in searching minimizers of

\begin{equation}  \label{ilambda}
I^{A}_\lambda \left( u\right) =\int_{\Omega }\left\vert \left( -\mathcal{L}%
\right) ^{\frac{s}{2}}u\right\vert ^{2} dx -\lambda \int_{\Omega }u^{2}dx
\end{equation}%
constrained to the Nehari manifold

\begin{equation*}
E=\left\{ u \in H^s: \int_{\Omega } |u|^{\frac{2n}{n-2s}}dx = 1\right\}\, .
\end{equation*}
In fact, the functional $I^{A}_\lambda$ is well defined on $H^s$ and its
least energy level on $E$, denoted by

\begin{equation*}
S^{A}_\lambda := \inf_{u \in E} I^{A}_\lambda \left( u\right)\, ,
\end{equation*}
is finite. Moreover, a minimizer $u \in H^s$ satisfies

\begin{equation}
\int_{\Omega }\left( -\mathcal{L}\right) ^{\frac{s}{2}}u\left( -\mathcal{L}%
\right) ^{\frac{s}{2}}\zeta dx=\int_{\Omega }\left( S^{A}_\lambda |u|^{\frac{%
4s}{n-2s}} u +\lambda u\right) \zeta dx  \label{integralform}
\end{equation}%
for every $\zeta \in H^{s}$. In particular, if $S^{A}_\lambda$ is positive
and $u$ is nonnegative in $\Omega$, then $u$ is a nonnegative weak solution
of (\ref{fractionalbn}). On the other hand, using the variational
characterization of the first Dirichlet eigenvalue of $(-\mathcal{L})^s$
given by

\begin{equation*}
\lambda_{1,s}\left(-\mathcal{L} \right) = \inf_{u \in H \setminus \{0\}}
\frac{\int_{\Omega }\left\vert \left( -\mathcal{L}\right) ^{\frac{s}{2}%
}u\right\vert ^{2} dx}{\int_{\Omega }u^{2}dx}\, ,
\end{equation*}
one can easily check that the positivity of $S^{A}_\lambda$ is equivalent to
$\lambda < \lambda_{1,s}\left(-\mathcal{L} \right)$. So, the conclusion of
Theorems \ref{T.1} and \ref{T.2} follows if we are able to prove the
existence and regularity of nonnegative minimizers of $I^{A}_\lambda$ in $E$
for any $\lambda > 0$.

Inspired by the recent work in \cite{ST}, an equivalent definition for the
operator $(- {\mathcal{L}})^{s}$ in $\Omega$ with zero Dirichlet boundary
condition can be formulated as an extension problem in a cylinder. Let $%
\mathcal{C}_{\Omega } =\Omega \times \left( 0,\infty \right) \subset \mathbb{%
R}_{+}^{n+1}$. We denote the points in $\mathcal{C}_{\Omega} $ by $z =
\left( x,y\right)$ with $x \in \Omega $ and the lateral boundary $\partial
\Omega \times \left[ 0,\infty \right)$ by $\partial _{L}\mathcal{C}_{\Omega }
$. Then for $u \in H^s$, we define the $(s,A)$-extension $w =
E_{s}^{A}\left( u\right)$ as the solution to the problem

\begin{equation}
\left\{
\begin{array}{rrrll}
\mathrm{div}\left( y^{1-2s}B\left( x\right) \nabla w\left( x,y\right) \right)
& = & 0 & \text{in} & \mathcal{C}_{\Omega }, \\
w & = & 0 & \text{on} & \partial _{L}\mathcal{C}_{\Omega }, \\
w & = & u & \text{on} & \Omega \times \left\{ y=0\right\} .%
\end{array}%
\right.  \label{extension w}
\end{equation}%
Here $B\left( x\right) =\left(
\begin{array}{cc}
A\left( x\right) & 0 \\
0 & 1%
\end{array}%
\right) $ is an $\left( n+1\right) \times \left( n+1\right) $ matrix. The
extension function belongs to the space%
\begin{equation*}
H_{0, L}^{s,A}\left( \mathcal{C}_{\Omega }\right) =\overline{C_{0}^{\infty
}\left( \Omega \times \left[ 0,\infty \right) \right) }^{\left\Vert \cdot
\right\Vert _{H_{0,L}^{s,A} }}
\end{equation*}%
with
\begin{equation*}
\left\Vert w\right\Vert _{H_{0,L}^{s,A} } = \left( c_{s}\iint_{\mathcal{C}%
_{\Omega }}y^{1-2s}\left( \nabla w\right) ^{T}B\left( x\right) \nabla
w\;dxdy\right) ^{1/2}.
\end{equation*}%
Here $c_{s}$ is a normalization constant such that $E_{s}^{A}:H^{s}\left(
\Omega \right) \rightarrow H_{0,L}^{s,A}\left( \mathcal{C}_{\Omega }\right) $
is an isometry between Hilbert spaces. In particular,
\begin{equation*}
\left\Vert E_{s}^{A}u\right\Vert _{H_{0,L}^{s,A} }=\left\Vert u\right\Vert
_{H^{s}}
\end{equation*}%
for every $u \in H^{s}$.

For $A\left( x\right) = I_n$, $E_{s}^{A}(u)$ is the canonical $s$-harmonic
extension of $u$, see \cite{CaSi}. In this case, we denote $E_{s}^{A}(u)$ by
$E_{s}\left( u\right) $ and the extension function space $%
H_{0,L}^{s,A}\left( \mathcal{C}_{\Omega }\right) $ is denoted by $%
H_{0,L}^{s}\left( \mathcal{C}_{\Omega }\right) $.

It is known from  \cite{ST} that the extension function $w$ satisfies
\begin{equation}
-c_{s}\lim_{y\rightarrow 0^{+}}y^{1-2s}\frac{\partial w}{\partial y}\left(
x,y\right) =\left( -\mathcal{L}\right) ^{s}u(x)  \label{neumanw}
\end{equation}%
for every $x\in \Omega $.

Using the extension map $E_{s}^{A}$, we can reformulate the problem $\left( %
\ref{fractionalbn}\right) $ as
\begin{equation}
\left\{
\begin{array}{rrlll}
-\mathrm{div}\left( y^{1-2s}B\left( x\right) \nabla w\right) & = & 0 & \text{%
in} & \mathcal{C}_{\Omega ,} \\
w & = & 0 & \text{on} & \partial \mathcal{C}_{\Omega }, \\
\partial^s_\nu w & = & w^{\frac{n+2s}{n-2s}}+\lambda w & \text{in} & \Omega
\times \left\{ y=0\right\} .%
\end{array}%
\right.  \label{fractionalbnw}
\end{equation}%
Here
\begin{equation*}
\partial^s_\nu w :=-c_{s}\lim_{y\rightarrow 0^{+}}y^{1-2s}\frac{\partial w}{%
\partial y}.
\end{equation*}

Using that the trace of functions in $H_{0,L}^{s,A}\left( \mathcal{C}%
_{\Omega }\right) $ is compactly embedded in $L^{q}(\Omega )$ for any $1\leq
q<\frac{2n}{n-2s}$ and continuously in $L^{\frac{2n}{n-2s}}(\Omega )$ for $%
n>2s$, we can consider the minimization of functional

\begin{equation}
J^{A}_\lambda \left( w\right) = c_{s} \iint_{\mathcal{C}_{\Omega
}}y^{1-2s}\left( \nabla w\right) ^{T}B\left( x\right) \nabla w\; dxdy-
\lambda \int_{\Omega }w^{2}\left( x,0\right) dx  \label{wenergy}
\end{equation}%
in the admissible set

\begin{equation*}
F=\left\{ w \in H_{0,L}^{s,A}\left( \mathcal{C}_{\Omega }\right):
\int_{\Omega } |w(x, 0)|^{\frac{2n}{n-2s}}dx = 1\right\}\, .
\end{equation*}
Clearly, we have

\begin{equation*}
S^{A}_\lambda = \inf_{w \in F} J^{A}_\lambda \left( w\right)\, .
\end{equation*}
Moreover, $u \in E$ is a minimizer of $I^{A}_\lambda$ on $E$ if, and only
if, $w = E_{s}^{A}(u) \in F$ is a minimizer of $J^{A}_\lambda$ on $F$.

There are essentially two advantages in considering the minimization problem
for $J_{\lambda}^A$. Firstly, it follows directly that $|w| \in F$ and $%
J^{A}_\lambda(|w|) = J^{A}_\lambda(w)$ for every $w \in F$, so that
minimizers of $J^{A}_\lambda$ on $F$ can be assumed nonnegative in $\Omega$.
Secondly, the integral
\begin{equation*}
\iint_{\mathcal{C}_{\Omega }}y^{1-2s}\left( \nabla w\right) ^{T}B\left(
x\right) \nabla w\; dxdy
\end{equation*}
is more easily to be handled comparing to the one in (\ref{ilambda}).
Therefore, from now on we will concentrate on the existence of minimizers of
$J^{A}_\lambda$ in $F$.

{ One of the tool used in our existence proof is the following
trace inequality
\begin{equation}
\left( \int_{\Omega }\left\vert f\left( x,0\right) \right\vert ^{r}
dx\right) ^{\frac{2}{r}} \leq A \int_{\mathcal{C}_{\Omega
}}y^{1-2s}\left\vert \nabla f\left( x,y\right) \right\vert ^{2}dxdy
\label{soboinequality}
\end{equation}%
for $1\leq r\leq \frac{2n}{n-2s},$ $n>2s$ and every $f\in H_{0}^{s}\left(
\mathcal{C}_{\Omega }\right)$. When $r=\frac{2n}{n-2s},$ we denote the best
constant in $\left( \ref{soboinequality}\right) $ by $K_{s}\left( n\right) .$
This constant is not achieved in bounded domain and achieved when $\Omega =%
\mathbb{R}^{n}$ and $f=E_{s}\left( u\right)$ with
\begin{equation}  \label{u tilde}
u(x)=\frac{\varepsilon ^{\frac{n-2s}{2}}}{\left( \left\vert x\right\vert
^{2}+\varepsilon ^{2}\right) ^{\frac{n-2s}{2}}} .
\end{equation}%
By a change of variable argument, we have
\begin{equation}
\left( \int_{\Omega }\left\vert f\left( x,0\right) dx\right\vert ^{\frac{2n}{%
n-2s}}\right) ^{\frac{n-2s}{n}}  \label{genesoboinequality}
\end{equation}%
}

\begin{equation*}
\leq \det \left( A\left( x_{0}\right) \right) ^{-\frac{s}{n}} K_{s}\left(
n\right) \int_{\mathcal{C}_{\Omega }}y^{1-2s}\left( \nabla f\left(
x,y\right) \right) ^{T}B\left( x_{0}\right) \nabla f\left( x,y\right) dxdy
\end{equation*}
for every $f\in H_{0}^{s}\left( \mathcal{C}_{\Omega }\right) .$ 

The following proposition states a necessary condition for existence of
minimizer of $I_{\lambda}^A$ in $E$.

\begin{proposition}
\label{existence tool} Let $n > 2s$. Assume there exists a point $x_0 \in
\overline{\Omega}$ such that (\ref{H2}) is satisfied and

\begin{equation}
S_{\lambda }^{A} < c_{s}\det \left( A\left( x_{0}\right) \right)^{\frac{s}{n}%
} K_s(n)^{-1}.  \label{S upperbound}
\end{equation}
Then, the infimum $S_{\lambda }^{A}$ of $I^A_\lambda$ in $E$ is achieved by
some nonnegative function $u$. Furthermore, if $\lambda <
\lambda_{1,s}\left(-\mathcal{L} \right)$, then $u$ is a nonnegative weak
solution of \eqref{fractionalbn}, module a suitable scaling.
\end{proposition}

\begin{proof}
As  noted above, it suffices to prove that the infimum of $J_{\lambda }^{A}$
in $F$, given also by $S_{\lambda }^{A}$, is assumed by some nonnegative
function $w$.

Let $\left\{ w_{m}\right\} \subset H_{0,L}^{s,A}\left( \mathcal{C}_{\Omega
}\right) $ be a minimizing sequence of $J_{\lambda }^{A}$ on $F$. Clearly, $%
\left\{ w_{m}\right\} $ is bounded in $H_{0,L}^{s,A}\left( \mathcal{C}%
_{\Omega }\right) $. Siince $\Omega $ is bounded, up to a subsequence, we
have

\begin{eqnarray*}
&&w_{m}\rightharpoonup w\text{ weakly in }H_{0,L}^{s,A}\left( \mathcal{C}%
_{\Omega }\right) , \\
&&w_{m}\left( \cdot ,0\right) \rightarrow w\left( \cdot .0\right) \text{
strongly in }L^{q}\left( \Omega \right) \text{ for }1\leq q<\frac{2n}{n-2s},
\\
&&w_{m}\left( \cdot ,0\right) \rightarrow w\left( \cdot ,0\right) \text{
a.e. in }\Omega .
\end{eqnarray*}%
A direct calculation, taking into account of the weak convergence, gives
\begin{eqnarray*}
\left\Vert w_{m}\right\Vert _{H_{0,L}^{s,A}}^{2} &=&\left\Vert
w_{m}-w\right\Vert _{H_{0,L}^{s,A}}^{2}+\left\Vert w\right\Vert
_{H_{0,L}^{s,A}}^{2} \\
&&+c_{s}\iint_{\mathcal{C}_{\Omega }}y^{1-2s}\left( \nabla w^{T}B\left(
x\right) \nabla \left( w_{m}-w\right) +\nabla ^{T}\left( w_{m}-w\right)
B\left( x\right) \nabla w\right) \;dxdy \\
&=&\left\Vert w_{m}-w\right\Vert _{H_{0,L}^{s,A}}^{2}+\left\Vert
w\right\Vert _{H_{0,L}^{s,A}}^{2}+o\left( 1\right) .
\end{eqnarray*}%
Now using Brezis-Lieb Lemma (see \cite{BrLi}), (\ref{H2}) and (\ref%
{soboinequality}), we obtain
\begin{eqnarray*}
S_{\lambda }^{A} &=&\left\Vert w_{m}\right\Vert _{H_{0,L}^{s,A}}^{2}-\lambda
\left\Vert w_{m}\left( \cdot ,0\right) \right\Vert _{L^{2}}^{2}+o(1) \\
&=&\left\Vert w_{m}-w\right\Vert _{H_{0,L}^{s,A}}^{2}+\left\Vert
w\right\Vert _{H_{0,L}^{s,A}}^{2}-\lambda \left\Vert w\left( \cdot ,0\right)
\right\Vert _{L^{2}}^{2}+o\left( 1\right)  \\
&\geq &c_{s}\iint_{\mathcal{C}_{\Omega }}y^{1-2s}\left( \nabla
(w_{m}-w)\right) ^{T}B\left( x_{0}\right) \nabla (w_{m}-w)\;dxdy+J_{\lambda
}^{A}(w)+o(1) \\
&\geq &c_{s}\det \left( A\left( x_{0}\right) \right) ^{\frac{s}{n}%
}K_{s}\left( n\right) ^{-1}\left\Vert \left( w_{m}-w\right) \left( \cdot
,0\right) \right\Vert _{L^{\frac{2n}{n-2s}}}^{2}+J_{\lambda }^{A}(w)+o\left(
1\right)  \\
&\geq &c_{s}\det \left( A\left( x_{0}\right) \right) ^{\frac{s}{n}%
}K_{s}\left( n\right) ^{-1}\left( 1-\int_{\Omega }|w(x,0)|^{\frac{2n}{n-2s}%
}dx\right) ^{\frac{n-2s}{n}} \\
&&+S_{\lambda }^{A}\left( \int_{\Omega }|w(x,0)|^{\frac{2n}{n-2s}}dx\right)
^{\frac{n-2s}{n}}+o(1)\,
\end{eqnarray*}%
so that

\begin{equation*}
S^A_\lambda \left( 1 - \left(\int_{\Omega } |w(x, 0)|^{\frac{2n}{n-2s}}dx
\right)^{\frac{n-2s}{n}}\right) \geq c_{s} \det \left( A\left( x_{0}\right)
\right)^{\frac{s}{n}} K_s(n)^{-1} \left( 1 - \int_{\Omega } |w(x, 0)|^{\frac{%
2n}{n-2s}}dx\right)^{\frac{n-2s}{n}}\, .
\end{equation*}
So, using the assumption (\ref{S upperbound}) and the fact that

\begin{equation*}
\int_{\Omega } |w(x, 0)|^{\frac{2n}{n-2s}}dx \leq 1\, ,
\end{equation*}
we derive

\begin{equation*}
\int_{\Omega } |w(x, 0)|^{\frac{2n}{n-2s}}dx = 1\, .
\end{equation*}
But this implies that

\begin{equation*}
w_{m}\left( \cdot ,0\right) \rightarrow w\left( \cdot .0\right) \text{
strongly in }L^{\frac{2n}{n-2s}}\left( \Omega \right) .
\end{equation*}%
Then, $w \in F$ and by lower semicontinuity of $J^A_\lambda$,
\begin{equation*}
J_{\lambda }^{A}\left( w\right) \leq \lim \inf_{m\rightarrow \infty
}J_{\lambda }^{A}\left( w_{m}\right) \, ,
\end{equation*}%
so that $J_{\lambda }^{A}\left( w\right) = S_{\lambda}^{A}$. Therefore, $|w|$
is a nonnegative minimizer of $J_{\lambda }^{A}$ in $F$.

The remainder of the proof is direct because $u = |w|(\cdot, 0)$ is a
nonnegative minimizer of $I_{\lambda }^{A}$ on $E$ and the inequality $%
\lambda < \lambda_{1,s}\left(-\mathcal{L} \right)$ is equivalent to the
positivity of $S_{\lambda }^{A}$.
\end{proof}

The discussion made so far about the existence of nonnegative weak solutions
of \eqref{fractionalbn} can be resumed in the following remark:

\begin{remark}
\label{R} Proposition \ref{existence tool} provides the existence of a
nonnegative weak solution of \eqref{fractionalbn} by assuming the conditions
(\ref{H2}), (\ref{S upperbound}) and $\lambda < \lambda_{1,s}\left(-\mathcal{%
L} \right)$. Up to a scaling this solution can be seen as the trace of a
nonnegative minimizer of $Q_{\lambda }^{A}$ on $H_{0,L}^{s,A}\left( \mathcal{%
C}_{\Omega }\right) \setminus \{0\}$, where $Q_{\lambda }^{A}$ denotes the
Rayleigh quotient
\begin{equation*}
Q_{\lambda }^{A}\left( w\right) =\frac{\left\Vert w\right\Vert
_{H_{0,L}^{s,A} }^{2}-\lambda \left\Vert w\left( \cdot,0\right) \right\Vert
_{L^{2} }^{2}}{\left\Vert w\left( \cdot,0\right) \right\Vert _{L^{\frac{2n}{%
n-2s}} }^{2}}\, .
\end{equation*}
Note that (\ref{S upperbound}) is equivalent to existence of a function $w_0
\in H_{0,L}^{s,A}\left( \mathcal{C}_{\Omega }\right) \setminus \{0\}$ so that

\begin{equation}
Q_{\lambda }^{A}\left( w_{0}\right) <c_{s}\det \left( A\left( x_{0}\right)
\right) ^{\frac{s}{n}}K_{s}\left( n\right) ^{-1}.  \label{Q upperbound}
\end{equation}%
So, in light of Proposition \ref{existence tool}, Sections 3 and 4 are
dedicated to the construction of $w_{0}$ by using the remaining assumptions
assumed in Theorems \ref{T.1} and \ref{T.2}, respectively.
\end{remark}

The next proposition shows further regularity of weak solution of %
\eqref{fractionalbn}. {}

\begin{proposition}
\label{regularity} Let $u\in H^{s}\setminus \{0\}$ be a nonnegative weak
solution of \eqref{fractionalbn}. Assume that $A(x)$ is a positive definite
symmetric matrix for each $x\in \overline{\Omega }$ with continuous entries
on $\overline{\Omega }$. Then $u\in L^{p}(\Omega )$ for every $p\geq 1$.
Furthermore, if $\partial \Omega $ is of $C^{1,1}$ class and each entry of $%
A(x)\in C^{1}\left( \overline{\Omega }\right) $ then $u$ belongs to $%
C^{0,\alpha }(\overline{\Omega })$ if $0<s<1/2$ and to $C^{1,\alpha }(%
\overline{\Omega })$ if $1/2\leq s<1$ for any $0<\alpha <1$.
\end{proposition}

\begin{proof}
The function $w=E_{s}^{A}(u)$ satisfies \eqref{fractionalbnw}. Since $u$ is
assumed to be nonnegative, $w$ is also nonnegative. For each $k\geq 1$, we
define $w_{k}$ by
\begin{equation*}
w_{k}(x,y):=\min \{w(x,y),k\}.
\end{equation*}%
Since $ww_{k}^{2\beta }$ is in $H_{0,L}^{s,A}(\mathcal{C}_{\Omega })$ for
all $\beta \geq 0$, using it as a test function in \eqref{fractionalbnw}, we
obtain
\begin{eqnarray}
&& \int_{\Omega } f(u)uu_{k}^{2\beta }\;dx  \label{times test} \\
&{=}& \iint_{\mathcal{C}_{\Omega }}y^{1-2s}(\nabla
w)^{T}B(x)\nabla (ww_{k}^{2\beta })\;dxdy  \notag \\
&=& \iint_{\mathcal{C}_{\Omega }}y^{1-2s}w_{k}^{2\beta }(\nabla w)^{T}B(x)
\nabla w +2\beta \;y^{1-2s}w_{k}^{2\beta }(\nabla w_{k})^{T}B(x) \nabla
w_{k}\;dxdy ,  \notag
\end{eqnarray}%
where $f(u) := u^{\frac{n + 2s}{n - 2s}} + \lambda u$. Note that the last
equality comes from the fact the $w=w_{k}$ in the set where $w\leq k$ and $%
w_{k}$ is a constant otherwise.

\noindent On the other hand, we have

\begin{eqnarray}
&&\iint_{\mathcal{C}_{\Omega }}y^{1-2s}(\nabla (ww_{k}^{\beta
}))^{T}B(x)\nabla (ww_{k}^{\beta })\;dxdy  \label{test2} \\
&=&\iint_{\mathcal{C}_{\Omega }}y^{1-2s}w_{k}^{2\beta }(\nabla
w)^{T}B(x)\nabla w+(2\beta +\beta ^{2})y^{1-2s}w_{k}^{2\beta }(\nabla
w_{k})^{T}B(x)\nabla w_{k}\;dxdy.  \notag
\end{eqnarray}%
Combining $\left( \ref{times test}\right) $ and $\left( \ref{test2}\right) $%
, we derive
\begin{equation}
\iint_{\mathcal{C}_{\Omega }}y^{1-2s}|\nabla (ww_{k}^{\beta
})|^{2}\;dxdy\leq C_{1}\int_{\Omega }|f(u)|uu_{k}^{2\beta }\;dx
\label{test ineq}
\end{equation}%
for some constant $C_{1}>0$ which depends only on $\beta $ and the matrix $A$%
. Let $h=|f(u)|/(1+u)$ and $\Omega _{m}=\{x\in \Omega :u(x)>m\}$.

Since  $h\in L^{\frac{n}{2s}}(\Omega )$, there exists $m\in \mathbb{N}$
large enough such that
\begin{equation*}
\left( \int_{\Omega _{m}}|h|^{\frac{n}{2s}}\;dx\right) ^{\frac{2s}{n}}\leq
\frac{K_{s}(n)}{4C_{1}},
\end{equation*}%
where $K_{s}(n)$ is the best constant with respect to the embedding $%
H_{0,L}^{s}(\mathcal{C}_{\Omega })\hookrightarrow L^{\frac{2n}{n-2s}}(\Omega
)$. Since $u$ is bounded on $\Omega \setminus \Omega _{m}$, there exists a
constant $C_{2}>0$ which depends only on $m$ and $\Omega $ such that
\begin{align}
\int_{\Omega }|f(u)|uu_{k}^{2\beta }\;dx=& \int_{\Omega \setminus \Omega
_{m}}|f(u)|uu_{k}^{2\beta }\;dx+\int_{\Omega _{m}}|f(u)|uu_{k}^{2\beta }\;dx
\notag \\
\leq & C_{2}+2\int_{\Omega _{m}}hu^{2}u_{k}^{2\beta }\;dx  \label{bd f}
\end{align}%
Using $\left( \ref{bd f}\right) $, we deduce that
\begin{align*}
\int_{\Omega }|f(u)|uu_{k}^{2\beta }\;dx\leq & C_{2}+2\left( \int_{\Omega
_{m}}h^{\frac{n}{2s}}\;dx\right) ^{\frac{2s}{n}}\left( \int_{\Omega
_{m}}(uu_{k}^{\beta })^{\frac{2n}{n-2s}}\;dx\right) ^{\frac{n-2s}{n}} \\
\leq & C_{2}+\frac{1}{2C_{1}}\iint_{\mathcal{C}_{\Omega }}y^{1-2s}|\nabla
(ww_{k}^{\beta })|^{2}\;dxdy.
\end{align*}%
Plugging this in $\left( \ref{test ineq}\right) $ and taking $k\rightarrow
\infty $ we have $u^{\beta +1}\in L^{\frac{2n}{n-2s}}(\Omega )$ for all $%
\beta \geq 0$. Thus $f(u)\in L^{p}(\Omega )$ for every $p\geq 1$ and the
rest of the proof follows from Theorems 1.1, 1.2, 1.3 and 1.5 of \cite{CS}.
\end{proof}

\section{Proof of Theorem \protect\ref{T.1} - Interior case}

Throughout this section, we assume $x_{0}$ is an interior point of $\Omega$.
According to Remark \ref{R} and Proposition \ref{regularity} of the previous
section, our main task in this section is to prove the following proposition:

\begin{proposition}
\label{laststep1} Assume (\ref{H1}) for some $x_0 \in \Omega$. If $n \geq 4s$
and $\sigma > 2s$, then for any $\lambda > 0$ there exists $w_0 \in
H_{0,L}^{s,A}\left( \mathcal{C}_{\Omega }\right) \setminus \{0\}$ such that

\begin{equation*}
Q_{\lambda }^{A}\left( w_{0}\right) <c_{s}\det \left( A\left( x_{0}\right)
\right) ^{\frac{s}{n}}K_{s}\left( n\right) ^{-1}.
\end{equation*}
\end{proposition}

We first derive some necessary estimate. For simplicity of notations, we
first assume $x_{0}=0$ and $A\left( 0\right) =I_{n}.$ Choose a smooth
nonincreasing cut-off function $\phi (t)\in C^{\infty }\left( \mathbb{R}%
_{+}\right) $ such that
\begin{equation*}
\phi (t)=1\;\text{ for }0\leq t\leq \frac{1}{2}\;\text{ and }\;\phi (t)=0\;%
\text{ if }t\geq 1.
\end{equation*}%
Let $r$ be small enough so that $B\left( 0,r\right) \subset \Omega .$ Define
$\phi _{r}\left( x,y\right) =\phi \left( \frac{r_{xy}}{r}\right) $ with $%
r_{xy}=\left( \left\vert x\right\vert ^{2}+y^{2}\right) ^{\frac{1}{2}}$. Let
$u_{\varepsilon }$ be given by \eqref{u tilde} and $w_{\varepsilon
}=E_{s}\left( u_{\varepsilon }\right) $. Then $w_{\varepsilon }\left(
x,y\right) =\varepsilon ^{n}w_{1}\left( \frac{x}{\varepsilon },\frac{y}{%
\varepsilon }\right) .$

\begin{lemma}
With the above notations, the family $\left\{ \phi_r w_{\varepsilon
}\right\}_{\varepsilon>0} $ and its trace on $\left\{ y=0\right\} ,$ namely $%
\left\{ \phi_r u_{\varepsilon }\right\}_{\varepsilon > 0} $ satisfy{%

\begin{eqnarray}
\left\Vert \phi_r w_{\varepsilon }\right\Vert _{H_{0,L}^{s,A}\left( \mathcal{%
C}_{\Omega }\right) }^{2} &\leq& c_{s}\int_{\mathbb{R}_{+}^{n+1}}y^{1-2s}\left\vert \nabla w_{\varepsilon
}(x,y)\right\vert ^{2}dxdy \label{phiwbound}\\
&&+\left\{
\begin{array}{ll}
O\left( \varepsilon ^{\sigma }\right) +O\left(
\varepsilon ^{n-2s}\right) , & \text{\ if }\sigma <n-2s \\
O\left( \varepsilon ^{\sigma }\ln \frac{1}{%
\varepsilon }\right) +O\left( \varepsilon ^{n-2s}\right) , & \text{\ if \ }%
\sigma =n-2s \\
O\left( \varepsilon ^{n-2s}\right) , & \text{\ if \ }%
\sigma >n-2s%
\end{array}%
\right. \nonumber
\end{eqnarray}%
}
\begin{equation}
\left\Vert \phi_r w_{\varepsilon }(x,0)\right\Vert _{L^{2}\left( \Omega
\right) }^{2}=\left\{
\begin{array}{ll}
C\varepsilon ^{2s}+O\left( \varepsilon ^{n-2s}\right) & \text{if }n>4s, \\
C\varepsilon ^{2s}\ln \frac{1}{\varepsilon }+O\left( \varepsilon ^{2s}\right)
& \text{if }n=4s.%
\end{array}%
\right.  \label{wepsilonl2}
\end{equation}
\end{lemma}

\begin{proof}
The equation $\left( \ref{wepsilonl2}\right) $ follows from Lemma 3.8 in
\cite{BCPS}. We only need to prove $\left( \ref{phiwbound}\right) .$ Since
\begin{eqnarray}
\left\Vert \phi _{r}w_{\varepsilon }\right\Vert _{H_{0,L}^{s,A}\left(
\mathcal{C}_{\Omega }\right) }^{2} &=&c_{s}\iint_{\mathcal{C}_{\Omega
}}y^{1-2s}\left[ \nabla \left( \phi _{r}w_{\varepsilon }\right) \right]
^{T}B\left( x\right) \nabla \left( \phi _{r}w_{\varepsilon }\right) dxdy
\notag \\
&=&c_{s}\iint_{\mathcal{C}_{\Omega }}y^{1-2s}\phi _{r}^{2}\left( \nabla
w_{\varepsilon }\right) ^{T}B\left( x\right) \nabla w_{\varepsilon }dxdy
\notag \\
&&+2c_{s}\iint_{\mathcal{C}_{\Omega }}y^{1-2s}\phi _{r}w_{\varepsilon
}\left( \nabla \phi _{r}\right) ^{T}B\left( x\right) \nabla w_{\varepsilon
}dxdy  \notag \\
&&+c_{s}\iint_{\mathcal{C}_{\Omega }}y^{1-2s}w_{\varepsilon }^{2}\left(
\nabla \phi _{r}\right) ^{T}B\left( x\right) \nabla \phi dxdy  \notag \\
&\leq &c_{s}\iint_{\mathcal{C}_{\Omega }}y^{1-2s}\phi _{r}^{2}\left( \nabla
w_{\varepsilon }\right) ^{T}B\left( 0\right) \nabla w_{\varepsilon }dxdy
\notag \\
&&+c_{s}\iint_{\mathcal{C}_{\Omega }}y^{1-2s}\phi _{r}^{2}\left\vert \nabla
w_{\varepsilon }\right\vert ^{2}\left\vert x\right\vert ^{\sigma }dxdy
\notag \\
&&+Cc_{s}\iint_{\mathcal{C}_{\Omega }}y^{1-2s}\left\vert \phi
_{r}\right\vert \left\vert \nabla \phi _{r}\right\vert \left\vert
w_{\varepsilon }\right\vert \left\vert \nabla w_{\varepsilon }\right\vert
dxdy  \notag \\
&&+Cc_{s}\iint_{\mathcal{C}_{\Omega }}y^{1-2s}\left\vert \nabla \phi
_{r}\right\vert ^{2}\left\vert w_{\varepsilon }\right\vert ^{2}dxdy  \notag
\\
&=&I+II+III+IV.  \label{xbound}
\end{eqnarray}%
The third and fourth term in $\left( \ref{xbound}\right) $ can be estimated
in the same way as in proof of Lemma 3.8 in \cite{BCPS} and we have
\begin{equation*}
III+IV\leq O\left( \varepsilon ^{n-2s}\right) .
\end{equation*}%
The first term in $\left( \ref{xbound}\right) $ can be estimated as%
\begin{equation}
c_{s}\iint_{\mathcal{C}_{\Omega }}y^{1-2s}\phi _{r}^{2}\left( \nabla
w_{\varepsilon }\right) ^{T}B\left( 0\right) \nabla w_{\varepsilon }dxdy\leq
c_{s}\int_{\mathbb{R}_{+}^{n+1}}y^{1-2s}\left\vert \nabla w_{\varepsilon
}(x,y)\right\vert ^{2}dxdy.  \label{scaledharmonic}
\end{equation}%
For the rest of the proof, we estimate the second term in $\left( \ref%
{xbound}\right) $. We shall show
\begin{equation}
\iint_{\mathcal{C}_{\Omega }}y^{1-2s}\phi _{r}^{2}\left\vert \nabla
w_{\varepsilon }\right\vert ^{2}\left\vert x\right\vert ^{\sigma }dxdy\leq
\left\{
\begin{array}{ll}
O\left( \varepsilon ^{\sigma }\right)  & \text{if }\sigma <n-2s, \\
O\left( \varepsilon ^{\sigma }\ln \frac{1}{\varepsilon }\right)  & \text{if }%
\sigma =n-2s, \\
O\left( \varepsilon ^{n-2s}\right)  & \text{if }\sigma >n-2s.%
\end{array}%
\right.   \label{sigmabound}
\end{equation}%
To prove $\left( \ref{sigmabound}\right) $, we write
\begin{eqnarray}
&&\iint_{\mathcal{C}_{\Omega }}y^{1-2s}\phi _{r}^{2}\left\vert \nabla
w_{\varepsilon }\right\vert ^{2}\left\vert x\right\vert ^{\sigma }dxdy
\notag \\
&\leq &\varepsilon ^{2s-n-2}\iint_{\left\{ r_{xy}\leq r\right\}
}y^{1-2s}\left\vert \nabla w_{1}\left( \frac{x}{\varepsilon },\frac{y}{%
\varepsilon }\right) \right\vert ^{2}\left\vert x\right\vert ^{\sigma }dxdy
\notag \\
&=&\varepsilon ^{\sigma }\iint_{\left\{ r_{xy}\leq \frac{r}{\varepsilon }%
\right\} }y^{1-2s}\left\vert x\right\vert ^{\sigma }\left\vert \nabla
w_{1}\left( x,y\right) \right\vert ^{2}dxdy  \notag \\
&\leq &C\varepsilon ^{\sigma }\iint_{\left\{ r_{xy}\leq \frac{r}{\varepsilon
}\right\} \cap \{|x|<1\}}y^{1-2s}\left\vert \nabla w_{1}\left( {x},y\right)
\right\vert ^{2}dxdy  \notag \\
&&+C\varepsilon ^{\sigma }\iint_{\left\{ r_{xy}\leq \frac{r}{\varepsilon }%
\right\} \cap \{|x|\geq 1\}}y^{1-2s}\left\vert x\right\vert ^{\sigma
}\left\vert \nabla w_{1}\left( x,y\right) \right\vert ^{2}d{x}dy.
\label{sigmasplit}
\end{eqnarray}%
We shall prove
\begin{eqnarray}
&&\iint_{\left\{ \left\vert x\right\vert \geq 1\right\} \cap \left\{
r_{xy}\leq \frac{r}{\varepsilon }\right\} }y^{1-2s}\left\vert x\right\vert
^{\sigma }\left\vert \nabla w_{1}\left( x,y\right) \right\vert ^{2}dxdy
\notag \\
&\leq &\left\{
\begin{array}{ll}
C & \text{if }\sigma <n-2s, \\
C\ln \frac{1}{\varepsilon } & \text{if }\sigma =n-2s, \\
C\varepsilon ^{n-2s-\sigma } & \text{if }\sigma >n-2s.%
\end{array}%
\right.   \label{sigmabigx}
\end{eqnarray}%
Then $\left( \ref{sigmabound}\right) $ follows directly from $\left( \ref%
{sigmasplit}\right) $ and $\left( \ref{sigmabigx}\right) $. We estimate $%
\nabla _{x}w_{1}\left( x,y\right) $ using its explicit representation
formula. Since
\begin{eqnarray*}
\nabla _{x_{i}}w_{1}\left( x,y\right)  &=&\int_{\mathbb{R}^{n}}\frac{%
y^{2s}\left( x_{i}-\xi _{i}\right) }{\left( y^{2}+\left\vert x-\xi
\right\vert ^{2}\right) ^{\frac{n+2s}{2}+1}\left( 1+\left\vert \xi
\right\vert ^{2}\right) ^{\frac{n-2s}{2}}}d\xi  \\
&=&\int_{\left\vert \xi \right\vert <\frac{\left\vert x\right\vert }{2}}%
\frac{y^{2s}\left( x_{i}-\xi _{i}\right) }{\left( y^{2}+\left\vert x-\xi
\right\vert ^{2}\right) ^{\frac{n+2s}{2}+1}\left( 1+\left\vert \xi
\right\vert ^{2}\right) ^{\frac{n-2s}{2}}}d\xi  \\
&&+\int_{\left\vert \xi \right\vert >\frac{3\left\vert x\right\vert }{2}}%
\frac{y^{2s}\left( x_{i}-\xi _{i}\right) }{\left( y^{2}+\left\vert x-\xi
\right\vert ^{2}\right) ^{\frac{n+2s}{2}+1}\left( 1+\left\vert \xi
\right\vert ^{2}\right) ^{\frac{n-2s}{2}}}d\xi  \\
&&+\int_{\frac{\left\vert x\right\vert }{2}<\left\vert \xi \right\vert <%
\frac{3\left\vert x\right\vert }{2}}\frac{y^{2s}\left( x_{i}-\xi _{i}\right)
}{\left( y^{2}+\left\vert x-\xi \right\vert ^{2}\right) ^{\frac{n+2s}{2}%
+1}\left( 1+\left\vert \xi \right\vert ^{2}\right) ^{\frac{n-2s}{2}}}d\xi  \\
&=&I_{1}+I_{2}+I_{3}
\end{eqnarray*}%
If $\left\vert \xi \right\vert <\frac{\left\vert x\right\vert }{2}$ or $%
\left\vert \xi \right\vert $ $>\frac{3\left\vert x\right\vert }{2},$ we have
$\left\vert x-\xi \right\vert >\frac{\left\vert x\right\vert }{2}$ and $%
\left\vert x-\xi \right\vert >\min \left( 1,\frac{1}{3}\right) \left\vert
\xi \right\vert .$ We can bound $I_{1}$ and $I_{2}$ as follows.
\begin{eqnarray}
\left\vert I_{1}\right\vert  &=&\left\vert \int_{\left\vert \xi \right\vert <%
\frac{\left\vert x\right\vert }{2}}\frac{y^{2s}\left( x_{i}-\xi _{i}\right)
}{\left( y^{2}+\left\vert x-\xi \right\vert ^{2}\right) ^{\frac{n+2s}{2}%
+1}\left( 1+\left\vert \xi \right\vert ^{2}\right) ^{\frac{n-2s}{2}}}d\xi
\right\vert   \notag \\
&\leq &C\frac{y^{-1+2s}}{\left( y^{2}+\left\vert x\right\vert ^{2}\right) ^{%
\frac{n+2s}{2}}}\int_{\left\vert \xi \right\vert <\frac{\left\vert
x\right\vert }{2}}\frac{1}{\left( 1+\left\vert \xi \right\vert ^{2}\right) ^{%
\frac{n-2s}{2}}}d\xi   \notag \\
&\leq &C\frac{y^{-1+2s}}{\left( y^{2}+\left\vert x\right\vert ^{2}\right) ^{%
\frac{n+2s}{2}}}\left\vert x\right\vert ^{2s},  \label{I1bound}
\end{eqnarray}%
\begin{eqnarray}
\left\vert I_{2}\right\vert  &=&\left\vert \int_{\left\vert \xi \right\vert >%
\frac{3\left\vert x\right\vert }{2}}\frac{y^{2s}\left( x_{i}-\xi _{i}\right)
}{\left( y^{2}+\left\vert x-\xi \right\vert ^{2}\right) ^{\frac{n+2s}{2}%
+1}\left( 1+\left\vert \xi \right\vert ^{2}\right) ^{\frac{n-2s}{2}}}d\xi
\right\vert   \notag \\
&\leq &C\frac{y^{-1+2s}}{\left( y^{2}+\left\vert x\right\vert ^{2}\right) ^{%
\frac{n-\delta }{2}}}\int_{\left\vert \xi \right\vert >\frac{3\left\vert
x\right\vert }{2}}\frac{1}{\left( 1+\left\vert \xi \right\vert ^{2}\right) ^{%
\frac{n-2s}{2}}\left\vert \xi \right\vert ^{2s+\delta }}d\xi   \notag \\
&\leq &C\frac{y^{-1+2s}}{\left( y^{2}+\left\vert x\right\vert ^{2}\right) ^{%
\frac{n-\delta }{2}}}\left\vert x\right\vert ^{-\delta }  \label{I2bound}
\end{eqnarray}%
Lastly, we bound $I_{3}$ in two different cases. If $y\geq \left\vert
x\right\vert ,$ we bound $I_{3}$ by
\begin{eqnarray}
\left\vert I_{3}\right\vert  &=&\left\vert \int_{\frac{\left\vert
x\right\vert }{2}<\left\vert \xi \right\vert <\frac{3\left\vert x\right\vert
}{2}}\frac{y^{2s}\left( x_{i}-\xi _{i}\right) }{\left( y^{2}+\left\vert
x-\xi \right\vert ^{2}\right) ^{\frac{n+2s}{2}+1}\left( 1+\left\vert \xi
\right\vert ^{2}\right) ^{\frac{n-2s}{2}}}d\xi \right\vert   \notag \\
&\leq &\frac{Cy^{2s-\frac{n+2s}{2}-1}}{\left( 1+\left\vert x\right\vert
^{2}\right) ^{\frac{n-2s}{2}}}\int_{\frac{\left\vert x\right\vert }{2}%
<\left\vert \xi \right\vert <\frac{3\left\vert x\right\vert }{2}}\left\vert
x-\xi \right\vert ^{-\frac{n+2s}{2}}d\xi   \notag \\
&\leq &\frac{Cy^{-\frac{n-2s}{2}-1}}{\left( 1+\left\vert x\right\vert
^{2}\right) ^{\frac{N-2s}{2}}}\int_{\left\vert \xi -x\right\vert <\frac{%
5\left\vert x\right\vert }{2}}\left\vert x-\xi \right\vert ^{-\frac{n+2s}{2}%
}d\xi   \notag \\
&\leq &C\frac{y^{-\frac{n-2s}{2}-1}}{\left( 1+\left\vert x\right\vert
^{2}\right) ^{\frac{n-2s}{2}}}\left\vert x\right\vert ^{\frac{n-2s}{2}}
\notag \\
&\leq &Cy^{-\frac{n-2s}{2}-1}\left\vert x\right\vert ^{-\frac{n-2s}{2}},
\label{I3boundybig}
\end{eqnarray}%
if $y\leq \left\vert x\right\vert ,$ we write
\begin{eqnarray*}
I_{3} &=&\int_{\frac{\left\vert x\right\vert }{2}<\left\vert \xi \right\vert
<\frac{3\left\vert x\right\vert }{2}}D_{x_{i}}P_{y}^{s}\left( x-\xi \right)
f\left( \xi \right) d\xi  \\
&=&-\int_{\frac{\left\vert x\right\vert }{2}<\left\vert \xi \right\vert <%
\frac{3\left\vert x\right\vert }{2}}D_{\xi _{i}}P_{y}^{s}\left( x-\xi
\right) f\left( \xi \right) d\xi  \\
&=&\int_{\frac{\left\vert x\right\vert }{2}<\left\vert \xi \right\vert <%
\frac{3\left\vert x\right\vert }{2}}P_{y}^{s}\left( x-\xi \right) D_{\xi
_{i}}f\left( \xi \right) d\xi  \\
&&-\int_{\left\{ \left\vert \xi \right\vert =\frac{\left\vert x\right\vert }{%
2}\right\} \cup \left\{ \left\vert \xi \right\vert =\frac{3\left\vert
x\right\vert }{2}\right\} }P_{y}^{s}\left( x-\xi \right) f\left( \xi \right)
dS_{\xi } \\
&=&J_{1}+J_{2}
\end{eqnarray*}%
\begin{eqnarray}
\left\vert J_{1}\right\vert  &=&\left\vert \int_{\frac{\left\vert
x\right\vert }{2}<\left\vert \xi \right\vert <\frac{3\left\vert x\right\vert
}{2}}P_{y}^{s}\left( x-\xi \right) D_{\xi _{i}}f\left( \xi \right) d\xi
\right\vert   \notag \\
&\leq &\frac{C}{\left( 1+\left\vert x\right\vert ^{2}\right) ^{\frac{n-2s+1}{%
2}}}\int_{\frac{\left\vert x\right\vert }{2}<\left\vert \xi \right\vert <%
\frac{3\left\vert x\right\vert }{2}}P_{y}^{s}\left( x-\xi \right) d\xi
\notag \\
&\leq &\frac{C}{\left( 1+\left\vert x\right\vert ^{2}\right) ^{\frac{n-2s+1}{%
2}}}  \label{I3boundinterior}
\end{eqnarray}%
When $\left\vert \xi \right\vert =\frac{\left\vert x\right\vert }{2}$ or $%
\frac{3\left\vert x\right\vert }{2},$%
\begin{equation}
P_{y}^{s}\left( x-\xi \right) f\left( \xi \right) \leq C\frac{y^{2s}}{\left(
y^{2}+\left\vert x\right\vert ^{2}\right) ^{\frac{n+2s}{2}}}\frac{1}{\left(
1+\left\vert x\right\vert ^{2}\right) ^{\frac{n-2s}{2}}}  \label{I3boundry}
\end{equation}%
When $\sigma <n-2s,$we have the following bound from $\left( \ref{I1bound}%
\right) $
\begin{eqnarray}
&&\int_{\left\vert x\right\vert \geq 1}\int_{\left\vert x\right\vert
}^{\infty }y^{1-2s}\left\vert x\right\vert ^{\sigma }I_{1}^{2}dxdy  \notag \\
&\leq &C\int_{\left\vert x\right\vert \geq 1}\int_{\left\vert x\right\vert
}^{\infty }y^{1-2s}\left\vert x\right\vert ^{\sigma }\frac{y^{2\left(
2s-1\right) }}{\left( y^{2}+\left\vert x\right\vert ^{2}\right) ^{n+2s}}%
\left\vert x\right\vert ^{4s}dydx  \notag \\
&\leq &C\int_{\left\vert x\right\vert \geq 1}\left\vert x\right\vert
^{\sigma +4s}\int_{\left\vert x\right\vert }^{\infty }y^{1-2s}\frac{%
y^{2\left( 2s-1\right) }}{\left( y\left\vert x\right\vert \right) ^{n+2s}}%
dydx  \notag \\
&\leq &C\int_{\left\vert x\right\vert \geq 1}\int_{\left\vert x\right\vert
}^{\infty }\left\vert x\right\vert ^{\sigma -n+2s}y^{-n-1}dydx  \notag \\
&\leq &C\int_{\left\vert x\right\vert \geq 1}\left\vert x\right\vert
^{\sigma -2n+2s}dx  \notag \\
&=&C\int_{1}^{\infty }r^{\sigma -n+2s-1}dr<\infty ,
\label{I1squareboundbigy}
\end{eqnarray}%
and
\begin{eqnarray}
&&\int_{\left\vert x\right\vert \geq 1}\int_{0}^{\left\vert x\right\vert
}y^{1-2s}\left\vert x\right\vert ^{\sigma }I_{1}^{2}dxdy  \notag \\
&\leq &C\int_{\left\vert x\right\vert \geq 1}\left\vert x\right\vert
^{4s+\sigma }\int_{0}^{\left\vert x\right\vert }y^{1-2s}\frac{y^{2\left(
2s-1\right) }}{\left( y^{2}+\left\vert x\right\vert ^{2}\right) ^{n+2s}}dydx
\notag \\
&\leq &C\int_{\left\vert x\right\vert \geq 1}\left\vert x\right\vert
^{\sigma +4s-2n-4s}\int_{0}^{\left\vert x\right\vert }y^{2s-1}dydx  \notag \\
&=&C\int_{1}^{\infty }r^{\sigma -n+2s-1}dr<\infty .
\label{I1sqareboundsmally}
\end{eqnarray}%
Similarly, $\left( \ref{I2bound}\right) $ implies
\begin{eqnarray}
&&\int_{\left\vert x\right\vert \geq 1}\int_{0}^{\infty }y^{1-2s}\left\vert
x\right\vert ^{\sigma }I_{2}^{2}dxdy  \notag \\
&\leq &C\int_{\left\vert x\right\vert \geq 1}\int_{0}^{\infty
}y^{1-2s}\left\vert x\right\vert ^{\sigma }\frac{y^{2\left( 2s-1\right) }}{%
\left( y^{2}+\left\vert x\right\vert ^{2}\right) ^{n-\delta }}\left\vert
x\right\vert ^{-2\delta }dydx  \notag \\
&=&C\int_{\left\vert x\right\vert \geq 1}\left\vert x\right\vert ^{\sigma
-2\delta }\int_{0}^{\infty }\frac{y^{2s-1}}{\left( y^{2}+\left\vert
x\right\vert ^{2}\right) ^{n-\delta }}dydx  \notag \\
&\leq &C\int_{\left\vert x\right\vert \geq 1}\left\vert x\right\vert
^{\sigma -2\delta -2n+2\delta +2s-1+1}\int_{0}^{\infty }\frac{u^{2s-1}}{%
\left( 1+u^{2}\right) ^{n-\delta }}dudx  \notag \\
&\leq &C\int_{1}^{\infty }r^{\sigma -n+2s-1}dr<\infty .
\label{I2squarebound}
\end{eqnarray}%
Lastly, it follows from $\left( \ref{I3boundybig}\right) $ that
\begin{eqnarray}
&&\int_{\left\vert x\right\vert \geq 1}\int_{\left\vert x\right\vert
}^{\infty }y^{1-2s}\left\vert x\right\vert ^{\sigma }I_{3}^{2}dxdy  \notag \\
&\leq &C\int_{\left\vert x\right\vert \geq 1}\left\vert x\right\vert
^{\sigma -\left( n-2s\right) }\int_{\left\vert x\right\vert }^{\infty
}y^{1-2s-\left( n-2s\right) -2}dydx  \notag \\
&\leq &C\int_{\left\vert x\right\vert \geq 1}\left\vert x\right\vert
^{\sigma -\left( 2n-2s\right) }dx  \notag \\
&=&C\int_{1}^{\infty }r^{\sigma -n+2s-1}dr<\infty .
\label{I3squareboundbigy}
\end{eqnarray}%
If $\sigma <n-2s,$ $\left( \ref{I3boundinterior}\right) $ and $\left( \ref%
{I3boundry}\right) $ imply
\begin{eqnarray}
&&\int_{\left\vert x\right\vert \geq 1}\int_{0}^{\left\vert x\right\vert
}y^{1-2s}\left\vert x\right\vert ^{\sigma }J_{1}^{2}dxdy  \notag \\
&\leq &C\int_{\left\vert x\right\vert \geq 1}\frac{\left\vert x\right\vert
^{\sigma }}{\left( 1+\left\vert x\right\vert ^{2}\right) ^{n-2s+1}}%
\int_{0}^{\left\vert x\right\vert }y^{1-2s}dydx  \notag \\
&\leq &C\int_{\left\vert x\right\vert \geq 1}\left\vert x\right\vert
^{\sigma -2\left( n-2s+1\right) }\left\vert x\right\vert ^{2-2s}dx  \notag \\
&=&C\int_{1}^{\infty }r^{\sigma -n+2s-1}dr<\infty ,
\label{I3squareboundsmally1}
\end{eqnarray}%
and
\begin{eqnarray}
&&\int_{\left\vert x\right\vert \geq 1}\int_{0}^{\left\vert x\right\vert
}y^{1-2s}\left\vert x\right\vert ^{\sigma }J_{2}^{2}dxdy  \notag \\
&\leq &C\int_{\left\vert x\right\vert \geq 1}\left\vert x\right\vert
^{\sigma }\int_{0}^{\left\vert x\right\vert }y^{1-2s}\frac{y^{4s}\left\vert
x\right\vert ^{2\left( n-1\right) }}{\left( y^{2}+\left\vert x\right\vert
^{2}\right) ^{n+2s}\left( 1+\left\vert x\right\vert ^{2}\right) ^{n-2s}}dydx
\notag \\
&\leq &C\int_{\left\vert x\right\vert \geq 1}\left\vert x\right\vert
^{\sigma +2\left( n-1\right) -4n}\int_{0}^{\left\vert x\right\vert
}y^{1+2s}dydx  \notag \\
&=&C\int_{1}^{\infty }r^{\sigma -n+2s-1}dr<\infty .
\label{I3squareboundsmally2}
\end{eqnarray}%
It then follows from $\left( \ref{I1squareboundbigy}\right) -\left( \ref%
{I3squareboundsmally2}\right) $ that
\begin{eqnarray*}
&&\int_{\left\vert x\right\vert \geq 1}\int_{0}^{\infty }y^{1-2s}\left\vert
x\right\vert ^{\sigma }\left\vert \nabla _{x}w_{1}\left( x,y\right)
\right\vert ^{2}dydx \\
&=&\int_{\left\vert x\right\vert \geq 1}\int_{0}^{\infty }y^{1-2s}\left\vert
x\right\vert ^{\sigma }\left( I_{1}+I_{2}+I_{3}\right) ^{2}dydx \\
&\leq &C\int_{\left\vert x\right\vert \geq 1}\int_{0}^{\infty
}y^{1-2s}\left\vert x\right\vert ^{\sigma }\left(
I_{1}^{2}+I_{2}^{2}+I_{3}^{2}\right) dydx<\infty .
\end{eqnarray*}%
For $\sigma =n-2s,$ the same integrals in $\left( \ref{I1squareboundbigy}%
\right) -\left( \ref{I3squareboundsmally2}\right) $ derive
\begin{equation*}
\iint_{\{\left\vert x\right\vert \geq 1\}\cap \left\{ {r_{xy}}\leq \frac{r}{%
\varepsilon }\right\} }y^{1-2s}\left\vert x\right\vert ^{\sigma }\left\vert
\nabla w_{1}\left( x,y\right) \right\vert ^{2}dydx\leq C\ln \frac{1}{%
\varepsilon }.
\end{equation*}%
For $\sigma >n-2s,$ integrals in $\left( \ref{I1squareboundbigy}\right)
-\left( \ref{I3squareboundsmally2}\right) $ yield
\begin{equation*}
\iint_{\{\left\vert x\right\vert \geq 1\}\cap \left\{ {r_{xy}}\leq \frac{r}{%
\varepsilon }\right\} }y^{1-2s}\left\vert x\right\vert ^{\sigma }\left\vert
\nabla w_{1}\left( x,y\right) \right\vert ^{2}dydx\leq C\left( \frac{1}{%
\varepsilon }\right) ^{\sigma -n+2s}=C\varepsilon ^{n-2s-\sigma }.
\end{equation*}
\end{proof}

For general $A\left( 0\right) $ case, we consider the following coordinate
transformation. Let $\left\{ a_{i}\right\} _{i=1}^{n}$ be eigenvalues of $%
A\left( x_{0}\right) $ and $O$ be the orthogonal matrix such that
\begin{equation*}
A\left(0\right) =O^{T} \mbox{diag }\left( a_{1},\ldots ,a_{n}\right) O.
\end{equation*}
Define the mapping $\Phi :\Omega \rightarrow $ $\widetilde{\Omega }=\Phi
\left( \Omega \right) $ by $\tilde{x}=\Phi \left( x\right) $ such that $%
\tilde{x}_{i}=a_{i}^{-\frac{1}{2}}\left( Ox\right) _{i}$. Let $\widetilde{u}%
_{\varepsilon }\left( \tilde{x}\right) $=$\frac{\varepsilon ^{\frac{n-2s}{2}}%
}{\left( \widetilde{\left\vert x\right\vert }^{2}+\varepsilon ^{2}\right) ^{%
\frac{n-2s}{2}}}$ and $\widetilde{w_{\varepsilon }}\left( \widetilde{x}%
,y\right) =E_{s}\left( \widetilde{u_{\varepsilon }}\right) $. Then we have
\begin{equation*}
\widetilde{w_{\varepsilon }}\left( \tilde{x},y\right) =\varepsilon ^{\frac{%
2s-n}{2}}\widetilde{w_{1}}\left( \frac{\tilde{x}}{\varepsilon },\frac{y}{%
\varepsilon }\right) .
\end{equation*}%
For $x=\Phi ^{-1}\left( \tilde{x}\right) $, we define $u_{\varepsilon
}\left( x\right) =\widetilde{u_{\varepsilon }}\left( \tilde{x}\right) $ and $%
w_{\varepsilon }\left( x,y\right) =\widetilde{w_{\varepsilon }}\left( \tilde{%
x},y\right) $.

To construct test functions, denote $B^{n+1}(x,r)$ the ball in ${\mathbb{R}}%
^{n+1}$ centered at $(x,0)$ with the radius $r$. We fix $r>0$ small enough
such that $\overline{B^{n+1}(\tilde{x}_{0},r)}\subset \overline{\mathcal{C}}%
_{\widetilde{\Omega }}$.We define the cutoff function $\phi _{r}(x,y)\in
C_{0}^{\infty }(\Phi ^{-1}(B_{+}^{n+1}(\tilde{x}_{0},r)))$ by $\phi
_{r}(x,y)=\phi _{0}(r_{xy}/r)$ with $r_{xy}^{2}=|\Phi (x)|^{2}+y^{2}$, and
let
\begin{equation}
V_{\varepsilon} (x,y)=\phi _{r}(x,y)w_{\varepsilon }(x,y).
\label{test function}
\end{equation}%
Under these notations, we have%
\begin{eqnarray*}
&&c_{s}\iint_{\mathcal{C}_{\Omega }}y^{1-2s}\phi _{r}^{2}\left( \nabla
w_{\varepsilon }\right) ^{T}B\left( 0\right) \nabla w_{\varepsilon }dxdy \\
&=&c_{s}\iint_{\mathcal{C}_{\Omega }}y^{1-2s}\phi _{r}^{2}\left(
\sum_{i=1}^{n}a_{i}\left( D_{x_{i}}w_{\varepsilon }\right) ^{2}+\left(
D_{y}w_{\varepsilon }\right) ^{2}\right) dxdy \\
&=&\det \left( A\left( 0\right) \right) ^{\frac{1}{2}}c_{s}\int_{\mathbb{R}%
_{+}^{n+1}}y^{1-2s}\left\vert \nabla _{\left( \tilde{x},y\right) }\widetilde{%
w_{\varepsilon }}(\tilde{x},y)\right\vert ^{2}d\widetilde{x}dy,
\end{eqnarray*}%
\begin{equation*}
\left\Vert u_{\varepsilon }\right\Vert _{L^{2}\left( \Omega \right)
}^{2}=\det \left( A\left( 0\right) \right) \left\Vert \widetilde{%
u_{\varepsilon }}\right\Vert _{L^{2}\left( \widetilde{\Omega }\right) }^{2},
\end{equation*}%
and
\begin{equation*}
\left\Vert u_{\varepsilon }\right\Vert _{L^{r}\left( \Omega \right)
}^{r}=\det \left( A\left( 0\right) \right) ^{\frac{r}{2}}\left\Vert
\widetilde{u_{\varepsilon }}\right\Vert _{L^{r}\left( \widetilde{\Omega }%
\right) }^{r}.
\end{equation*}%
We then have the following bounds on $V_{\varepsilon .}$

\begin{lemma}
\label{Z general estimate} With the above notations, for small $\varepsilon
>0$, the family of functions $\left\{ V_{\varepsilon }\right\} _{\varepsilon
>0}$ and its trace on $\left\{ y=0\right\} ,$ namely $\left\{ V_{\varepsilon
}(x,0)\right\} _{\varepsilon >0}$ satisfy{
\begin{eqnarray}
\left\Vert V_{\varepsilon }\right\Vert _{H_{0,L}^{s,A}\left( \mathcal{C}%
_{\Omega }\right) }^{2} &\leq& c_{s}\int_{\mathbb{R}_{+}^{n+1}}y^{1-2s}\left\vert \nabla \tilde{w_{\varepsilon
}}(x,y)\right\vert ^{2}d\tilde{x}dy\\  \label{phiw}
&&+ \left\{
\begin{array}{ll}
O\left( \varepsilon ^{\sigma }\right)
+O\left( \varepsilon ^{n-2s}\right) , & \text{if }\sigma <n-2s \\
O\left( \varepsilon ^{\sigma }\ln \frac{1}{%
\varepsilon }\right) +O\left( \varepsilon ^{n-2s}\right) , & \text{if }%
\sigma =n-2s \\
O\left( \varepsilon ^{n-2s}\right) , &
\text{if }\sigma >n-2s,%
\end{array}%
\right. \nonumber
\end{eqnarray}%
}
\begin{equation}
\left\Vert V_{\varepsilon }(x,0)\right\Vert _{L^{2}\left( \Omega \right)
}^{2}=\left\{
\begin{array}{ll}
C\varepsilon ^{2s}+O\left( \varepsilon ^{n-2s}\right) & \text{if }n>4s, \\
C\varepsilon ^{2s}\ln \frac{1}{\varepsilon }+O\left( \varepsilon ^{2s}\right)
& \text{if }n=4s,%
\end{array}%
\right. .  \label{uepsilonl2}
\end{equation}
\end{lemma}

\noindent \textit{{Proof of Proposition \ref{laststep1}}.} Fix $r$ small
such that $V_{\varepsilon }=\phi _{r}w_{\varepsilon }$ defined by %
\eqref{test function} is in $H_{0,L}^{s,A}\left( \Omega \right) $. Recall
that $\tilde{x}_{i}=a_{i}^{-\frac{1}{2}}\left( Ox\right) _{i}$ and $%
\widetilde{u_{\varepsilon }}\left( \widetilde{x}\right) =u_{\varepsilon
}\left( x\right) $. We have
\begin{equation*}
\left\Vert u_{\varepsilon }\right\Vert _{L^{\frac{2n}{n-2s}}}^{2}=\det
\left( A\left( 0\right) \right) ^{\frac{n-2s}{n}}\left\Vert \widetilde{%
u_{\varepsilon }}\right\Vert _{L^{\frac{2n}{n-2s}}}^{2}.
\end{equation*}%
Let $K_{1}=\left\Vert \widetilde{u_{\varepsilon }}\right\Vert _{L^{\frac{2n}{%
n-2s}}}^{\frac{2n}{n-2s}}$. Then $K_{1}$ is independent of $\varepsilon $
and by calculations in the proof of Proposition 4.1 \cite{BCPS},
\begin{equation}
\int_{\Omega }\left\vert V_{\varepsilon }\right\vert ^{\frac{2n}{n-2s}%
}dx\geq K_{1}\det \left( A\left( 0\right) \right) ^{\frac{1}{2}}+O\left(
\varepsilon ^{n}\right) .  \label{Z ineq}
\end{equation}%
{ Since $\widetilde{w_{\varepsilon }}$ is an extremal function of
(\ref{genesoboinequality}), we have
\begin{equation*}
K_{1}^{-\frac{\left( n-2s\right) }{n}}\iint_{\mathbb{R}_{+}^{n+1}}y^{1-2s}%
\left\vert \nabla \widetilde{w_{\varepsilon }}\right\vert
^{2}dxdy=K_{s}\left( n\right) ^{-1}.
\end{equation*}%
When $n>4s$, if $2s<\sigma <n-2s,$
\begin{eqnarray}
Q_{\lambda }^{A}\left( V_{\varepsilon }\right)  &\leq &\frac{c_{s}(%
\mbox{det}(A(0)))^{\frac{1}{2}}\int_{\mathbb{R}_{+}^{n+1}}y^{1-2s}\left\vert
\nabla \widetilde{w_{\varepsilon }}\right\vert ^{2}dxdy-\lambda C\varepsilon
^{2s}+O\left( \varepsilon ^{n-2s}\right) +O\left( \varepsilon ^{\sigma
}\right) }{K_{1}^{\frac{\left( n-2s\right) }{n}}\det \left( A\left( 0\right)
\right) ^{\frac{n-2s}{2n}}+O\left( \varepsilon ^{n}\right) }  \notag \\
&\leq &\frac{c_{s}(\mbox{det}(A(0)))^{\frac{s}{n}}K_{s}\left( n\right)
^{-1}-\lambda C\varepsilon ^{2s}K_{1}^{-\frac{2\left( n-2s\right) }{\left(
n+2s\right) }}+O\left( \varepsilon ^{n-2s}\right) +O\left( \varepsilon
^{\sigma }\right) }{1+O\left( \varepsilon ^{n}\right) }  \notag \\
&<&c_{s}(\mbox{det}(A(0)))^{\frac{s}{n}}K_{s}(n)^{-1}  \label{qlambdabd}
\end{eqnarray}%
for $\varepsilon \ll 1$. If $\sigma \geq n-2s,$ we replace $\varepsilon
^{\sigma }$ by $\varepsilon ^{\sigma }\ln \frac{1}{\varepsilon }$ or $%
\varepsilon ^{n-2s}$ in $\left( \ref{qlambdabd}\right) $, then the same
estimate follows. When $n=2s,$ we replace $\varepsilon ^{2s}$ by $%
\varepsilon ^{2s}\ln \frac{1}{\varepsilon }$ in $\left( \ref{qlambdabd}%
\right) $, conclusion follows from the fact that $\sigma >2s.$}\ \qed\newline

\section{Proof of Theorem \protect\ref{T.2} - Boundary case}

{ Throughout this section, we assume $x_{0} = 0$ is on the
boundary of $\Omega $ and $\partial \Omega $ is $\alpha $-singular at $x_{0}$%
. Our main task is to prove (\ref{Q upperbound}) when $n>4s$, $\sigma >\frac{%
2s(n-2s)}{n-4s}$, and $1\leq \alpha <\frac{\sigma (n-4s)}{2s(n-2s)}$. }

We consider the mapping $\Phi: \Omega \rightarrow \widetilde {\Omega}$
defined in Section 3 and denote $\widetilde {x} = \Phi(x)$. Then by
Definition \ref{boundary regularity}, there exist a constant $\delta > 0$
and a sequence $(x_j) \subset \Omega$ (i.e. $(\widetilde{x}_j) \subset
\widetilde{\Omega}$) such that $x_j \to 0$ (i.e. $\widetilde{x}_j
\rightarrow 0$) as $j \rightarrow + \infty$ and $\Phi^{-1}( B(\widetilde{x}%
_j, \delta |\widetilde{x}_j |^\alpha) ) \subset \Omega $. Let
\begin{equation*}
V_{\varepsilon }\left( x,y\right) =\phi _{\delta }\left( x,y\right)
w_{\varepsilon }\left( x,y\right)
\end{equation*}%
defined as in $\left( \ref{test function}\right)$.  For fixed $\beta >\alpha$%
, we consider
\begin{equation*}
\widetilde{V}_{j}(x,y)=\phi _{\delta }(\frac{x-x_{j}}{\varepsilon
_{j}^{\alpha }},\frac{y}{\varepsilon _{j}^{\alpha }})\;w_{\varepsilon
_{j}^{\beta }}(x-x_{j},y),
\end{equation*}%
where $\varepsilon _{j}=|x_{j}-x_{0}|^{\alpha }$. Then $\widetilde{V}%
_{j}(x,y)$ can be rewritten as
\begin{eqnarray*}
\widetilde{V}_{j}(x,y) &=&\left( \varepsilon _{j}^{\alpha }\right) ^{\frac{%
2s-n}{2}}\phi _{\delta }(\frac{x-x_{j}}{\varepsilon _{j}^{\alpha }},\frac{y}{%
\varepsilon _{j}^{\alpha }})\;\left( \varepsilon _{j}^{\beta -\alpha
}\right) ^{\frac{2s-n}{2}}w_{1}(\frac{x-x_{j}}{\varepsilon _{j}^{\alpha
}\cdot \varepsilon _{j}^{\beta -\alpha }},\frac{y}{\varepsilon _{j}^{\alpha
}\cdot \varepsilon _{j}^{\beta -\alpha }}) \\
&=&\left( \varepsilon _{j}^{\alpha }\right) ^{\frac{2s-n}{2}} V_{\varepsilon
_{j}^{\beta -\alpha }}(\frac{x-x_{j}}{\varepsilon _{j}^{\alpha }},\frac{y}{%
\varepsilon _{j}^{\alpha }}).
\end{eqnarray*}%
Thus
\begin{equation*}
\nabla \widetilde{V}_{j}(x,y)=\varepsilon _{j}^{\frac{\alpha (2s-n-2)}{2}%
}\nabla V_{\varepsilon _{j}^{\beta -\alpha }}(\frac{x-x_{j}}{\varepsilon
_{j}^{\alpha }},\frac{y}{\varepsilon _{j}^{\alpha }}).
\end{equation*}%
Then, by change of variables, we have
\begin{eqnarray}
&& \int_{{\mathbb{R}}_{+}^{n+1}}y^{1-2s} \left( \nabla\widetilde{V}_{j}
\right)^T B(0) \nabla\widetilde{V}_{j} dxdy  \notag \\
&=& \int_{{\mathbb{R}} _{+}^{n+1}}y^{1-2s} \left( \nabla V_{\varepsilon
_{j}^{\beta -\alpha}}(x,y) \right)^T B(0) \nabla V_{\varepsilon _{j}^{\beta
-\alpha }}(x,y) dxdy.  \label{Z norm part 1}
\end{eqnarray}
The triangle inequality implies
\begin{equation*}
|x| \leq |x -x_j| + |x_j| \leq \delta \varepsilon_j^\alpha + \varepsilon_j.
\end{equation*}
Hence
\begin{equation}  \label{Z norm part 2}
\int_{{\mathbb{R}}_{+}^{n+1}}y^{1-2s}|x|^{\sigma }|\nabla \widetilde{V}%
_{j}|^{2} dxdy = O(\varepsilon_j^\sigma).
\end{equation}
Combining \eqref{Z norm part 1}, \eqref{Z norm part 2} and applying Lemma %
\ref{Z general estimate} we obtain
\begin{eqnarray*}
\left\Vert \widetilde{V}_{j}\right\Vert _{H_{0,L}^{s,A}\left( \mathcal{C}%
_{\Omega }\right) }^{2} &\leq& \left\Vert V_{\varepsilon _{j}^{\beta
-\alpha}} \right\Vert _{H_{0,L}^{s,A(0)}\left( \mathcal{C}_{\Omega }\right)
}^{2} + O(\varepsilon_j^\sigma) \\
&\leq& (\mbox{det} (A(0)))^{\frac{1}{2}} \left\Vert \widetilde{w}%
_{\varepsilon_j^{\beta - \alpha}} \right\Vert _{H_{0,L}^{s}\left( \mathcal{C}%
_{\widetilde{\Omega}}\right) }^{2} + O(\varepsilon_j^{(n-2s)(\beta -
\alpha)}) + O(\varepsilon_j^\sigma).
\end{eqnarray*}
Furthermore, by \eqref{uepsilonl2} and \eqref{Z ineq} we have
\begin{equation*}
\int_{{\mathbb{R}}^{n}}| \widetilde{V}_{j}(x,0)|^{\frac{2n}{n-2s}}dx=\int_{{%
\mathbb{R}}^{n}}| V_{\varepsilon _{j}^{\beta -\alpha }}(x,0)|^{\frac{2n}{n-2s%
}}dx \geq K_{1}+O(\varepsilon _{j}^{n(\beta -\alpha )}),
\end{equation*}%
and for $n>4s$,
\begin{eqnarray*}
\int_{{\mathbb{R}}^{n}}|\widetilde{V}_{j}(x,0)|^{2}dx &=&\varepsilon
_{j}^{2s\alpha }\int_{{\mathbb{R}}^{n}}|V_{\varepsilon _{j}^{\beta -\alpha
}}(x,0)|^{2}dx \\
&=&\varepsilon _{j}^{2s\alpha }\left[ C\varepsilon _{j}^{2s(\beta -\alpha
)}+O(\varepsilon _{j}^{(n-2s)(\beta -\alpha )})\right].
\end{eqnarray*}%
Repeating our argument in Proposition \ref{laststep1}, we obtain
\begin{eqnarray*}
Q_{\lambda }^{A}\left( \widetilde{V}_j\right) &\leq & \frac{c_{s} (\mbox{det}%
(A(x_0)))^{\frac{s}{n}} K_s\left( n\right)^{-1} -\lambda C\varepsilon
_{j}^{2s\beta } K_1^{-\frac{2(n-2s)}{n+2s}} +O(\varepsilon
_{j}^{(n-2s)(\beta -\alpha )})+O\left( \varepsilon _{j}^{\sigma }\right) }{%
1+O\left( \varepsilon ^{n}\right) }.  \label{qlambdabound}
\end{eqnarray*}%
By our assumption on $\sigma ,$ we can choose $\beta $ such that
\begin{equation}
\frac{\alpha \left( n-2s\right) }{n-4s}<\beta <\frac{\sigma }{2s}.
\label{beltachoice}
\end{equation}%
It then follows from $\left( \ref{beltachoice}\right) $ that
\begin{equation}
2s\beta <\min \left( \sigma ,\left( n-2s\right) \left( \beta -\alpha \right)
\right) .  \label{beltabound}
\end{equation}%
$\left( \ref{qlambdabound}\right) $ and $\left( \ref{beltabound}\right) $
yield
\begin{equation*}
Q_{\lambda }^{A}\left( \widetilde{V}_{j} \right) < c_{s} (\mbox{det}%
(A(x_0)))^{\frac{s}{n}} K_s\left( n\right)^{-1}.
\end{equation*}

\section{Proof of Theorem \protect\ref{T.3} - Nonexistence}

When $\lambda $ is nonpositive, our nonexistence result relies on the
following Pohozaev identity.

\begin{lemma}
\label{lemma pz} Assume $\partial \Omega \in C^{1}$ and $a_{ij}\in C^{1}(%
\overline{\Omega }\setminus \{x_{0}\})$. Let $A^{\prime
}(x)=(a_{ij}{}^{\prime }(x))$ where $a_{ij}^{\prime }(x):=\nabla
a_{ij}(x)\cdot (x-x_{0})$. Assume further that each $a_{ij}^{\prime }$
extends continuously to $x_{0}$. Then for $u\in C^{1}(\overline{\Omega })$
and $w=E_{A}^{s}(u)$, we have
\begin{eqnarray}
&&\frac{1}{2}\int_{\partial _{L}\mathcal{C}_{\Omega }}y^{1-2s}(\nabla
_{x}w)^{T}A(x)(\nabla _{x}w)(x-x_{0})\cdot \nu _{\Omega }\;d\sigma  \notag \\
&=&\frac{s}{c(s)}\lambda \int_{\Omega }u^{2}dx-\frac{1}{2}\iint_{\mathcal{C}%
_{\Omega }}y^{1-2s}(\nabla _{x}w)^{T}A^{\prime }(x)(\nabla _{x}w)\;dxdt,
\label{pz}
\end{eqnarray}%
where $\nu _{\Omega }$ is the outward normal of $\partial \Omega $, and $%
d\sigma $ is the area element of $\partial _{L}\mathcal{C}_{\Omega }$.
\end{lemma}

\begin{proof}
Without loss of generality, we assume $x_{0}=0$. From approximation
arguments in \cite{DMS}, it suffices to prove $\left( \ref{pz}\right) $ for
coefficients $a_{ij}\in C^{1}(\overline{\Omega })$ and functions $u\in C^{2}(%
\overline{\Omega })$. Let $z:=(x,y)\in \Omega \times {\mathbb{R}}_{+}$.
Since the matrix $B(x)=\left(
\begin{array}{cc}
A(x) & 0 \\
0 & 1%
\end{array}%
\right) $ is symmetric, we have
\begin{eqnarray}
&&{\text{div}}\left[ y^{1-2s}\left( z\cdot \nabla w\right) B\left( x\right)
\nabla w\right]   \notag \\
&=&(z\cdot \nabla w)\mathrm{\text{div}}\left[ y^{1-2s}B\left( x\right)
\nabla w\right] +\left[ y^{1-2s}B\left( x\right) \nabla w\right] ^{T}\nabla
\left( z\cdot \nabla w\right)   \notag \\
&=&(z\cdot \nabla w)\text{div}\left[ y^{1-2s}B\left( x\right) \nabla w\right]
{+y}^{1-2s}\left( \nabla w\right) ^{T}B\left( x\right) \nabla \left( z\cdot
\nabla w\right)   \label{div total}
\end{eqnarray}%
A direct calculation shows that
\begin{eqnarray}
&&\frac{1}{2}\nabla \lbrack (\nabla w)^{T}B(x)(\nabla w)]\cdot z  \notag \\
&=&(\nabla w)^{T}B(x)\nabla (\nabla w\cdot z)-(\nabla w)^{T}B(x)\nabla w+%
\frac{1}{2}(\nabla _{x}w)^{T}A^{\prime }(x)\nabla _{x}w.  \label{grad norm}
\end{eqnarray}%
Combining $\left( \ref{div total}\right) $ and $\left( \ref{grad norm}%
\right) $, we have%
\begin{eqnarray}
&&\text{div}\left[ y^{1-2s}\left( z\cdot \nabla w\right) B\left( x\right)
\nabla w\right]   \notag \\
&{=}&\left( z\cdot \nabla w\right) \text{div}\left[ y^{1-2s}B\left( x\right)
\nabla w\right]   \notag \\
&&+\frac{y^{1-2s}}{2}\nabla \left[ \left( \nabla w\right) ^{T}B\left(
x\right) \nabla w\right] \cdot z  \notag \\
&&+y^{1-2s}(\nabla w)^{T}B(x)\nabla w-\frac{y^{1-2s}}{2}(\nabla
_{x}w)^{T}A^{\prime }(x)\nabla _{x}w.  \label{div total 2}
\end{eqnarray}%
Integrating both sides of $\left( \ref{div total 2}\right) $ over the set $%
\mathcal{C}_{R,\varepsilon }:=\Omega \times (\varepsilon ,R)$ for fixed $%
R>\varepsilon >0$ we obtain%
\begin{eqnarray}
&&\int_{\partial {\mathcal{C}_{R,\varepsilon }}}y^{1-2s}(z\cdot \nabla
w)(\nabla w)^{T}B(x)\nu \;dS  \notag \\
&=&\iint_{\mathcal{C}_{R,\varepsilon }}(z\cdot \nabla w)\text{div}%
[y^{1-2s}B(x)\nabla w]+\frac{y^{1-2s}}{2}\nabla \lbrack (\nabla
w)^{T}B(x)(\nabla w)]\cdot z\;dxdy  \notag \\
&&+\iint_{\mathcal{C}_{R,\varepsilon }}y^{1-2s}(\nabla w)^{T}B(x)\nabla w-%
\frac{y^{1-2s}}{2}(\nabla _{x}w)^{T}A^{\prime }(x)\nabla _{x}w\;dxdy.
\label{int eq}
\end{eqnarray}%
The first term on the right hand side of $\left( \ref{int eq}\right) $ is
zero since $\text{div}\left[ y^{1-2s}B\left( x\right) \nabla w\right] =0$ in
$\Omega \times {\mathbb{R}}_{+}$. Integrating by parts for the second term
on the right hand side of $\left( \ref{int eq}\right) $ we derive
\begin{eqnarray}
&&\frac{1}{2}\iint_{\mathcal{C}_{R,\varepsilon }}y^{1-2s}\nabla \lbrack
(\nabla w)^{T}B(x)(\nabla w)]\cdot z\;dxdy  \notag \\
&=&-\frac{n+1}{2}\iint_{\mathcal{C}_{R,\varepsilon }}y^{1-2s}(\nabla
w)^{T}B(x)(\nabla w)\;dxdy  \notag \\
&&-\frac{1-2s}{2}\iint_{\mathcal{C}_{R,\varepsilon }}y^{-2s}\cdot y(\nabla
w)^{T}B(x)(\nabla w)\;dxdy  \notag \\
&&+\frac{1}{2}\int_{\partial \mathcal{C}_{R,\varepsilon }}y^{1-2s}(\nabla
w)^{T}B(x)(\nabla w)(z\cdot \nu )\;dS  \notag \\
&=&-\frac{n+2-2s}{2}\iint_{\mathcal{C}_{R,\varepsilon }}y^{1-2s}(\nabla
w)^{T}B(x)(\nabla w)\;dxdy  \notag \\
&&+\frac{1}{2}\int_{\partial \mathcal{C}_{R,\varepsilon }}y^{1-2s}(\nabla
w)^{T}B(x)(\nabla w)(z\cdot \nu )\;dS.  \label{int grad norm}
\end{eqnarray}%
The boundary integral in $\left( \ref{int grad norm}\right) $ can be written
into three terms:%
\begin{eqnarray}
&&\int_{\partial \mathcal{C}_{R,\varepsilon }}y^{1-2s}(\nabla
w)^{T}B(x)(\nabla w)(z\cdot \nu _{\Omega })\;dS  \notag \\
&=&\int_{\Omega \times \{y=R\}}y^{2-2s}(\nabla w)^{T}B(x)(\nabla w)\;dx
\notag \\
&&-\int_{\Omega \times \{y=\varepsilon \}}y^{2-2s}(\nabla w)^{T}B(x)(\nabla
w)\;dx  \notag \\
&&+\int_{\partial _{L}\mathcal{C}_{R,\varepsilon }}y^{1-2s}(\nabla
w)^{T}B(x)(\nabla w)(x\cdot \nu _{\Omega })\;d\sigma .  \label{right bd}
\end{eqnarray}%
As in $\left( \ref{right bd}\right) $, we write the left hand side of $%
\left( \ref{int eq}\right) $ into three parts:%
\begin{eqnarray}
&&\int_{\partial {\mathcal{C}_{R,\varepsilon }}}y^{1-2s}(z\cdot \nabla
w)(\nabla w)^{T}B(x)\nu \;dS  \notag \\
&=&\int_{\Omega \times \{y=R\}}y^{1-2s}(z\cdot \nabla w)\frac{\partial w}{%
\partial y}\;dx-\int_{\Omega \times \{y=\varepsilon \}}y^{1-2s}(z\cdot
\nabla w)\frac{\partial w}{\partial y}\;dx  \notag \\
&&+\int_{\partial _{L}{\mathcal{C}_{R,\varepsilon }}}y^{1-2s}(z\cdot \nabla
w)(\nabla w)^{T}B(x)\nu _{\Omega }\;d\sigma   \notag \\
&=&\int_{\Omega \times \{y=R\}}y^{1-2s}(z\cdot \nabla w)\frac{\partial w}{%
\partial y}\;dx-\int_{\Omega \times \{y=\varepsilon \}}y^{1-2s}(z\cdot
\nabla w)\frac{\partial w}{\partial y}\;dx  \notag \\
&&+\int_{\partial _{L}{\mathcal{C}_{R,\varepsilon }}}y^{1-2s}(\nabla
w)^{T}B(x)(\nabla w)(x\cdot \nu _{\Omega })\;d\sigma .  \label{left bd}
\end{eqnarray}%
Here the last equality comes from the fact that $-\nabla w/|\nabla w|=\nu
_{\Omega }$ on $\partial _{L}\mathcal{C}$. Combining $\left( \ref{int eq}%
\right) $-$\left( \ref{left bd}\right) $, we obtain

\begin{eqnarray}
&&\int_{\partial _{L}{\mathcal{C}_{R,\varepsilon }}}y^{1-2s}(\nabla
w)^{T}B(x)(\nabla w)(x\cdot \nu _{\Omega })\;d\sigma   \notag \\
&=&\int_{\Omega \times \{y=\varepsilon \}}y^{1-2s}(z\cdot \nabla w)\frac{%
\partial w}{\partial y}\;dx-\frac{1}{2}\int_{\Omega \times \{y=\varepsilon
\}}y^{2-2s}(\nabla w)^{T}B(x)(\nabla w)\;dx  \notag \\
&&+\frac{2s-n}{2}\iint_{\mathcal{C}_{R,\varepsilon }}y^{1-2s}(\nabla
w)^{T}B(x)(\nabla w)\;dxdy  \notag \\
&&-\frac{1}{2}\iint_{\mathcal{C}_{R,\varepsilon }}y^{1-2s}(\nabla
_{x}w)^{T}A^{\prime }(x)\nabla _{x}w\;dxdy  \notag \\
&&+\frac{1}{2}\int_{\Omega \times \{y=R\}}y^{2-2s}(\nabla w)^{T}B(x)(\nabla
w)\;dx-\int_{\Omega \times \{y=R\}}y^{1-2s}(z\cdot \nabla w)\frac{\partial w%
}{\partial y}\;dx.  \label{int eq2}
\end{eqnarray}%
The second term on the right hand side of $\left( \ref{int eq2}\right) $
approaches to zero as $\varepsilon $ does since $s<1$. For the last two
terms, there exists $C>0$ such that for all $R\geq 1$, we have
\begin{eqnarray}
&&\left\vert \frac{1}{2}\int_{\Omega \times \{y=R\}}y^{2-2s}(\nabla
w)^{T}B(x)(\nabla w)\;dx-\int_{\Omega \times \{y=R\}}y^{1-2s}(z\cdot \nabla
w)\frac{\partial w}{\partial y}\;dx\right\vert   \label{bd lim R} \\
&\leq &C\int_{\Omega \times \{y=R\}}R^{2-2s}|\nabla w|^{2}dx.  \notag
\end{eqnarray}%
We claim that there exists a sequence $\{R_{i}\}$ such that
\begin{equation*}
\int_{\Omega \times \{y=R_{i}\}}R_{i}^{2-2s}|\nabla w|^{2}dx\rightarrow
0\quad \text{as}\quad R_{i}\rightarrow \infty .
\end{equation*}%
Suppose by contradiction there exist $a_{0}>0$ and $R_{0}\geq 1$ such that
\begin{equation*}
\int_{\Omega \times \{t=R\}}R^{2-2s}|\nabla w|^{2}dx\geq a_{0}\quad \text{%
for all }R\geq R_{0}.
\end{equation*}%
Then for any $R\geq R_{0}$,
\begin{equation*}
\int_{{\mathbb{R}}_{+}}\int_{\Omega }y^{1-2s}|\nabla w|^{2}dxdy\geq
\int_{R_{0}}^{R}\frac{1}{y}\left[ \int_{\Omega }y^{2-2s}|\nabla w|^{2}dx%
\right] dy\geq a_{0}\ln \frac{R}{R_{0}}.
\end{equation*}%
This contradicts $w\in H_{0,L}^{s}(\mathcal{C}_{\Omega })$. By $\left( \ref%
{fractionalbnw}\right) ,$%
\begin{align*}
& \lim_{\varepsilon \rightarrow 0}\int_{\Omega \times \{y=\varepsilon
\}}y^{1-2s}(z\cdot \nabla w)\frac{\partial w}{\partial y}\;dx \\
=& -\frac{1}{c_{s}}\int_{\Omega \times \{y=0\}}(x\cdot \nabla _{x}w)(\lambda
w+|w|^{\frac{n+2s}{n-2s}})\;dx \\
=& -\frac{1}{4c_{s}}\int_{\Omega \times \{y=0\}}\nabla _{x}|x|^{2}\cdot
\nabla _{x}\left( \lambda w^{2}+\frac{n-2s}{n}|w|^{\frac{n+2s}{n-2s}%
+1}\right) \;dx \\
=& \frac{n}{2c_{s}}\int_{\Omega }\lambda u^{2}+\frac{n-2s}{n}|u|^{\frac{2n}{%
n-2s}}\;dx.
\end{align*}%
Furthermore,%
\begin{eqnarray}
&&\lim_{R\rightarrow \infty }\lim_{\varepsilon \rightarrow 0}\iint_{\mathcal{%
C}_{R,\varepsilon }}y^{1-2s}(\nabla w)^{T}B(x)(\nabla w)\;dxdy  \notag \\
&=&\iint_{\mathcal{C}_{\Omega }}y^{1-2s}(\nabla w)^{T}B(x)(\nabla w)dxdy
\notag \\
&=&\frac{1}{c_{s}}\int_{\Omega }\lambda u^{2}+|u|^{\frac{2n}{n-2s}}\;dx.
\label{interior lim}
\end{eqnarray}%
Taking $\varepsilon \rightarrow 0$ and $R=R_{i}\rightarrow \infty $ in $%
\left( \ref{int eq2}\right) $ the lemma follows from $\left( \ref{bd lim R}%
\right) $-$\left( \ref{interior lim}\right) $.
\end{proof}

\noindent \textit{Proof of Theorem \ref{T.3}.} Suppose by contradiction that
problem $\left( \ref{fractionalbn}\right) $ admits a positive solution $u$
in $C^{1}(\overline{\Omega })$. Then, its extension given by $%
w=E_{s}^{A}(u)\in C^{1}(\overline{\mathcal{C}}_{\Omega })$ is also positive
in $\mathcal{C}_{\Omega }$ and satisfies $w=0$ on $\partial _{L}\mathcal{C}%
_{\Omega }$. By the Hopf lemma (see for example \cite{GT}), $\nabla _{x}w$
is nonzero on $\partial _{L}\mathcal{C}_{\Omega }$. Note also that the
assumption of $\Omega $ is star-shaped implies $(x-x_{0})\cdot \nu _{\Omega
}>0$ for all $x\in \overline{\Omega }$. Hence, the left-hand side of $\left( %
\ref{pz}\right) $ is strictly positive. On the other hand, since $\lambda
\leq 0$ and $A^{\prime }(x)$ is positive semi-definite, the right-hand side
of $\left( \ref{pz}\right) $ is non-positive. This contradicts Lemma \ref%
{lemma pz}.\ \qed

\end{document}